\documentclass[12pt]{amsart}
\usepackage{latexsym}
\usepackage{amsfonts}
\usepackage{amsthm}
\usepackage{hyperref}

\usepackage{amscd,amssymb,amsmath,tikz}
\usepackage{enumerate}
\usepackage{enumitem}
\usepackage{mathrsfs}
\usepackage{wasysym}
\usepackage{mathtools}
\usepackage{tikz-cd}
\usepackage{subfigure}

\usetikzlibrary{arrows.meta,chains,matrix,decorations.pathreplacing}
\usetikzlibrary{patterns}

 \newtheorem{theorem}{Theorem}[section]
 \newtheorem{cor}[theorem]{Corollary}
 \newtheorem{lemma}[theorem]{Lemma}
 \newtheorem{proposition}[theorem]{Proposition}
 \theoremstyle{definition}
 
  \newtheorem{ex}[theorem]{Example}

 \theoremstyle{remark}
 \newtheorem{remark}[theorem]{Remark}
 \numberwithin{equation}{section}

\newcommand{\YL}{\mathcal{L}}

\DeclareMathOperator{\fs}{\mathfrak{F}}
\DeclareMathOperator{\key}{\kappa}
\DeclareMathOperator{\yqk}{\hat{\mathfrak{Q}}}
\DeclareMathOperator{\qk}{\mathfrak{Q}}
\DeclareMathOperator{\yqs}{\hat{\mathscr{S}}}
\DeclareMathOperator{\qs}{\mathscr{S}}
\DeclareMathOperator{\atom}{\mathcal{A}}
\DeclareMathOperator{\ya}{\hat{\mathcal{A}}}

\DeclareMathOperator{\std}{std}
\DeclareMathOperator{\yfs}{\hat{\mathfrak{F}}}
\DeclareMathOperator{\yms}{\hat{\mathfrak{M}}}
\DeclareMathOperator{\ykey}{\hat{\kappa}}

\DeclareMathOperator{\fp}{\mathfrak{L}}
\DeclareMathOperator{\yfp}{\hat{\mathfrak{L}}}
\DeclareMathOperator{\msp}{\mathfrak{M}}

\DeclareMathOperator{\sch}{\mathfrak{S}}
\DeclareMathOperator{\ysch}{\widehat{\mathfrak{S}}}
\DeclareMathOperator{\yflagschur}{\widehat{\mathfrak{FS}}}
\DeclareMathOperator{\ykeymod}{\hat{\mathcal{K}}}
\DeclareMathOperator{\flagschur}{\mathfrak{FS}}
\DeclareMathOperator{\keymod}{\mathcal{K}}
\DeclareMathOperator{\yw}{\hat{\mathcal{W}}}
\DeclareMathOperator{\w}{\mathcal{W}}

\DeclareMathOperator{\colform}{colform}

\DeclareMathOperator{\ch}{ch}

\DeclareMathOperator{\col}{col}

\newcommand{\SSYT}{\mathrm{SSYT}}
\newcommand{\YCT}{\mathrm{YCT}}
\newcommand{\FF}{\mathrm{FF}}
\newcommand{\YFF}{\mathrm{YFF}}
\newcommand{\MF}{\mathrm{MF}}
\newcommand{\YMF}{\mathrm{YMF}}
\newcommand{\ASSF}{\mathrm{ASSF}}
\newcommand{\YASSF}{\mathrm{YASSF}}
\newcommand{\KSSF}{\mathrm{KSSF}}
\newcommand{\YKSSF}{\mathrm{YKSSF}}
\newcommand{\LF}{\mathrm{LF}}
\newcommand{\YLF}{\mathrm{YLF}}

\newcommand{\RCT}{\mathrm{RCT}}

\newcommand{\QF}{\mathrm{QF}}

\newcommand{\YQF}{\mathrm{YQF}}

\newcommand{\PD}{\mathrm{PD}}
\newcommand{\YPD}{\mathrm{YPD}}

\newcommand{\RFYC}{\mathrm{RFYC}}
\newcommand{\RF}{\mathrm{RF}}

\newcommand{\frev}{\mathrm{frev}}
\newcommand{\maxcomp}{\mathrm{maxcomp}}

\newcommand{\keytab}{\mathrm{key}}
\newcommand{\comp}{\mathrm{comp}}

\newcommand{\Fill}{\mathrm{Fill}}

\newcommand{\Red}{\mathrm{Red}}

\newcommand{\wt}{{\rm wt}}
\newcommand{\ywt}{{\rm ywt}}
\newcommand{\rev}{{\rm rev}}
\newcommand{\flatten}{{\rm flat}}
\newcommand{\sort}{{\rm sort}}
\newcommand{\revsort}{{\rm revsort}}

\newcommand{\Sym}{\ensuremath{\mathrm{Sym}}}
\newcommand{\QSym}{\ensuremath{\mathrm{QSym}}}
\newcommand{\Poly}{\ensuremath{\mathrm{Poly}}}

\newcommand{\excise}[1]{}

\newlength{\cellsize}
\newcommand\tableau[1]{
\vcenter{
\let\\=\cr
\baselineskip=-16000pt
\lineskiplimit=16000pt
\lineskip=0pt
\halign{&\tableaucell{##}\cr#1\crcr}}}

\newcommand{\tableaucell}[1]{{%
\def \arg{#1}\def \void{}%
\ifx \void \arg
\vbox to \cellsize{\vfil \hrule width \cellsize height 0pt}%
\else
\unitlength=\cellsize
\begin{picture}(1,1)
\put(0,.22){\makebox(1,1)[b]{$#1$}}
\put(0,0){\line(1,0){1}}
\put(0,1){\line(1,0){1}}
\put(0,0){\line(0,1){1}}
\put(1,0){\line(0,1){1}}
\end{picture}%
\fi}}

\newcommand\elbow{
\begin{picture}(10,10)
\thicklines
\put(10,10){\oval(10,10)[bl]}
\put(0,0){\oval(10,10)[tr]}
\end{picture}}

\newcommand\upelb{
\begin{picture}(10,10)
\thicklines
\put(0,0){\oval(10,10)[tr]}
\end{picture}}

\newcommand\cross{
\begin{picture}(10,10)
\thicklines
\put(5,0){\line(0,1){10}}
\put(0,5){\line(1,0){10}}
\end{picture}}

\newcommand\gridify[1]{\vbox to 10\unitlength{\vss\hbox to 10\unitlength{\hss$_{#1}$\hss}\vss}}

\newcommand\pipes[1]{\vtop{\let\\=\cr
\setlength\baselineskip{-10000pt}
\setlength\lineskiplimit{10000pt}
\setlength\lineskip{0pt}
\halign{&\gridify{##}\cr#1\crcr}}}

\usepackage[colorinlistoftodos]{todonotes}

\begin{document}

\title{The ``Young" and ``reverse" dichotomy of polynomials}

\author[S. Mason]{Sarah Mason}
\address{Department of Mathematics, Wake Forest University, Winston-Salem, NC 27109, U.S.A.}
\email{masonsk@wfu.edu}

\author[D. Searles]{Dominic Searles}
\address{Department of Mathematics and Statistics, University of Otago, 730 Cumberland St., Dunedin 9016, New Zealand}
\email{dominic.searles@otago.ac.nz}


\date{\today}


\maketitle
\maketitle

\begin{abstract}
A ``flip-and-reversal'' involution arising in the study of quasisymmetric Schur functions provides a passage between what we term ``Young'' and ``reverse'' variants of bases of polynomials or quasisymmetric functions. Building on this perspective, which has found recent application in the study of $q$-analogues of combinatorial Hopf algebras and generalizations of dual immaculate functions, we develop and explore Young analogues of well-known bases for polynomials.  We prove several combinatorial formulas for the Young analogue of the key polynomials, show that they form the generating functions for left keys, and provide a representation-theoretic interpretation of Young key polynomials as traces on certain modules. We also give combinatorial formulas for the Young analogues of Schubert polynomials, including their crystal graph structure. We moreover determine the intersections of (reverse) bases and their Young counterparts, further clarifying their relationships to one another. 
\end{abstract}

\tableofcontents

\section{Introduction}

Tableau models provide an indispensable framework for giving explicit positive combinatorial formulas for important families of polynomials and their relationships to one another. The celebrated \emph{Schur polynomials}, which form a basis for the ring $\Sym_n$ of symmetric polynomials in $n$ variables, are famously realized as the weight generating functions of \emph{semistandard Young tableaux}: tableaux of partition shape whose entries weakly increase from left to right in each row and strictly increase from bottom to top in each column.  In fact, this definition may be reversed and Schur polynomials may alternatively be realized as the weight generating functions of \emph{semistandard reverse tableaux}, whose entries weakly decrease from left to right along rows rows and strictly decrease from bottom to top in each column. 

$\Sym_n$ is a subring of the ring $\QSym_n$ of quasisymmetric polynomials. Basis elements of $\QSym_n$ are indexed by \emph{compositions} (sequences of positive integers) with at most $n$ parts. The semistandard reverse tableau model used in $\Sym_n$ naturally extends to produce tableaux of composition shape. The \emph{diagram} $D(\alpha)$ of a composition $\alpha$, written in French notation, is the diagram consisting of left-justified rows of boxes whose $i^{th}$ row from the bottom contains $\alpha_i$ boxes.  A \emph{tableau} (of shape $\alpha$) is a filling of $D(\alpha)$ with positive integers.  A \emph{reverse composition tableau} is a tableau with entries no larger than $n$, so that entries \emph{weakly decrease} from left to right along rows. 

Imposing different choices of further restrictions on the entries produces collections of reverse composition tableaux whose weight generating functions are, for example, the \emph{quasisymmetric Schur polynomial} \cite{HLMvW09}, the \emph{fundamental quasisymmetric polynomial} \cite{Ges84}, or the \emph{monomial quasisymmetric polynomial} \cite{Ges84} corresponding to $\alpha$.  On the other hand, certain other bases of $\QSym_n$ are naturally described instead by restrictions of \emph{Young} composition tableaux, where entries \emph{weakly increase} from left to right along rows. Examples include the \emph{dual immaculate polynomials} \cite{BerBerSalSerZab14}, the \emph{Young quasisymmetric Schur polynomials} \cite{LMvWbook}, and the \emph{extended Schur polynomials} \cite{Assaf.Searles:3}.

Extending further, \emph{reverse fillings} provide a combinatorial framework that naturally generalizes the model of reverse composition tableaux to the ring $\Poly_n$ of polynomials in $n$ variables. Basis elements of $\Poly_n$ are indexed by \emph{weak compositions}: sequences of nonnegative integers.  The diagram $D(a)$ of a weak composition $a$ is the diagram in $\mathbb{N}\times \mathbb{N}$ having $a_i$ boxes in row $i$, left-justified.  A \emph{filling} (of shape $a$) is an assignment of positive integers, no larger than $n$, to the boxes of $D(a)$.  A reverse filling is a filling in which entries weakly decrease from left to right along each row.

By imposing further restrictions on the entries, one can obtain a set of reverse fillings of $D(a)$ whose weight generating function is, for example, the \emph{key polynomial}~\cite{RS95}, the \emph{quasi-key polynomial}~\cite{Assaf.Searles:2}, the \emph{Demazure atom}~\cite{Mas09}, or the \emph{fundamental slide polynomial}~\cite{Sea20} corresponding to $a$. At present, the majority of well-studied bases for $\Poly_n$ are described in terms of reverse fillings, i.e., with decreasing rows.

As noted earlier, Schur polynomials may be realized in terms of either semistandard Young tableaux or semistandard reverse tableaux. This coincidence can be understood in terms of an involution on tableaux whose entries are at most $n$, namely, replacing each entry $i$ with $n+1-i$.  This bijectively maps semistandard Young tableaux to semistandard reverse tableaux and vice versa. 
While this map is weight-reversing rather than preserving, the fact that Schur polynomials are symmetric means that the multiset of weights of semistandard Young tableaux is equal to the multiset of weights of semistandard reverse tableaux.

This map inspires a closely-related \emph{flip-and-reverse} map on composition tableaux, defined by reversing the order of the rows (\emph{reverse}) and replacing every entry $i$ with $n+1-i$ (\emph{flip}). This weight-reversing map changes decreasing rows to increasing rows and vice versa. As is the case for Schur polynomials, the flip-and-reverse map preserves both the monomial and fundamental bases of $\QSym_n$. However, bases of $\QSym_n$ are not preserved in general. In particular, the reverse composition tableaux that generate the quasisymmetric Schur polynomial corresponding to $\alpha$ are mapped to precisely the Young composition tableaux that generate the Young quasisymmetric Schur polynomial corresponding to $\rev(\alpha)$, the composition obtained by reading $\alpha$ in reverse. Typically a Young quasisymmetric Schur polynomial  is not also a quasisymmetric Schur polynomial; we characterize their coincidences in Section~\ref{sec:background}. 
  
The flip-and-reverse map extends naturally to fillings of weak composition diagrams, giving two parallel constructions of bases for $\Poly_n$, one (\emph{reverse}) defined by reverse fillings and one (\emph{Young}) defined by \emph{Young fillings}, i.e., fillings in which entries increase from left to right along rows.  The fillings obtained by applying the flip-and-reverse map to those reverse fillings that generate a particular basis of $\Poly_n$ generate a Young analogue of that basis.  

Young analogues of the quasi-key and fundamental slide bases and a reverse analogue of the dual immaculate functions were introduced in~\cite{MasSea20} and properties of these bases were developed including a number of useful applications. In particular, these analogues were used to extend a result of \cite{AHM18} on positive expansions of dual immaculate functions to the full polynomial ring, to establish properties of stable limits of these polynomials and their expansions, and to uncover a previously-unknown connection between dual immaculate functions and Demazure atoms. These results necessitated repeated passage between reverse and Young analogues.  In particular, reverse analogues were needed to study stable limits for a polynomial ring analogue of the dual immaculate functions, whereas Young analogues were needed to connect to established results in $\QSym_n$ from \cite{AHM18}. In a similar vein, Young analogues of pre-existing reverse bases of $\QSym_n$ were applied in the study of $q$-analogues of combinatorial Hopf algebras \cite{Li15} and skew variants of quasisymmetric bases \cite{MN-SkewRS} to take advantage of classical combinatorics in $\Sym_n$ concerning Schur functions and Young tableaux. 
This type of relabelling is also used in~\cite{PreRic21} (there called ``shifting") to simplify arguments relating to the equivariant cohomology of Springer fibers for $GL_n(\mathbb{C})$.

We are motivated by the utility of the flip-and-reverse perspective to explore and develop further Young analogues of bases of $\Poly_n$ and establish structural results.  The Young analogue of the key polynomials is of particular interest and forms a primary focus. In fact, this Young basis has already found application: this variant of the key polynomials is used in~\cite{HRS18} to obtain the Hilbert series of a generalization of the coinvariant algebra.  In Section 3 we establish a connection with left and right \emph{keys} of semistandard Young tableaux, proving in Theorem~\ref{thm:leftkeygen} that the Young key polynomials are in fact a generating function for semistandard Young tableaux whose left key is greater than a fixed key.  We establish an analogous result for the Young analogue of the Demazure atom basis.  We also provide a representation-theoretic construction for the Young key polynomials as traces of the action of a diagonal matrix on certain modules.  Moreover, in addition to the Young skyline filling model arising from the flip-and-reverse map, we detail several other constructions and interpretations of the Young key polynomials and Young atoms, including divided difference operators, crystal graphs, and compatible sequences.  

In Section~\ref{Sec:others} we provide a new formula for the expansion of a key polynomial into fundamental slide polynomials as well as a new combinatorial construction of the \emph{fundamental particle} basis for polynomials \cite{Sea20} in terms of \emph{flag-compatible sequences}.  We describe Young analogues for additional families of polynomials, classify which of these Young bases expand positively in one another, and explain different behaviour exhibited by Young and reverse versions including stable limits and embedding into larger polynomial rings. 
We also completely determine the intersection of the Young and reverse versions of all bases we consider. As a result, we find that when the Young and reverse versions of such a basis of $\Poly_n$ extend a given basis of $\Sym_n$ or $\QSym_n$, the intersection of the Young and reverse basis of $\Poly_n$ is exactly the original basis. For example, we show that the intersection of the Young key polynomials and the key polynomials is exactly the Schur polynomials, and the intersection of the fundamental slide and Young fundamental slide polynomials is exactly the fundamental quasisymmetric polynomials.

Finally in Section~\ref{Sec:Schubert}, we introduce a Young analogue of the famous Schubert polynomials, extending this perspective further. We describe how to generate the Young Schubert polynomials using pipe dreams and divided difference operators and detail how Young Schubert polynomials expand into Young key polynomials.  Interestingly, unlike the case for Young analogues of other polynomial bases, there is no basis of $\Poly_n$ consisting of Young Schubert polynomials.  We also describe the crystal graph structure for Young Schubert polynomials (analogous to the crystal graph structure for Young key polynomials), as Demazure subcrystals of the crystal on \emph{reduced factorizations} introduced in \cite{MorSch16}, using methods that were developed on a flipped and reversed version of this crystal in \cite{AssSch18}.

\section{Background}\label{sec:background}

Throughout the following, we denote permutations in one-line notation and allow the transposition $s_i$ to act on the right by swapping the entries in the $i$th and $(i+1)$th positions.  For a weak composition $a$, let $\sort(a)$ denote the partition obtained by recording the entries of $a$ in weakly decreasing order. We refer to assignments of integers to diagrams of compositions as \emph{tableaux} and assignments of integers to diagrams of weak compositions as \emph{fillings}. For any tableau or filling $T$, the \emph{weight} $\wt(T)$ denotes the weak composition whose $i$th entry is the number of occurrences of $i$ in $T$.

\subsection{Quasisymmetric polynomials}~\label{sec:qsymintro}

Let $\alpha$ be a composition with at most $n$ parts. The \emph{fundamental quasisymmetric polynomial} $F_\alpha(x_1,\ldots , x_n)$ was originally introduced through the enumeration of $P$-partitions~\cite{Ges84}.   Although there are several different ways to generate the fundamental quasisymmetric polynomials, we describe them as generating functions for certain tableau-like objects which we call \emph{fundamental reverse composition tableaux} to align with other definitions to follow. Fundamental reverse composition tableaux are those reverse composition tableaux (i.e., entries decrease from left to right in each row) satisfying the additional condition that if $i<j$, then every entry in row $i$ is strictly smaller than every entry in row $j$. It is straightforward to check that this definition is equivalent to the definition of the fundamental quasisymmetric polynomials as generating functions of ribbon tableaux (see for example~\cite[Section 4.1]{Hua16}).

In this way, $F_\alpha(x_1, \ldots , x_n)$ is the sum of all monomials $x^{\wt(T)}$, where $T$ ranges over fundamental reverse composition tableaux of shape $\alpha$ and largest entry at most $n$.

\begin{ex}
We have $F_{13}(x_1,x_2,x_3) = x^{013}+x^{103}+x^{112} + x^{121}+x^{130}$, as witnessed by the following fundamental reverse composition tableaux.
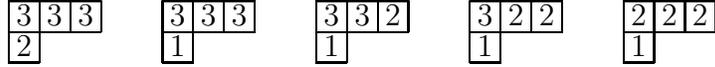
\begin{displaymath}
 \begin{array}{c@{\hskip2\cellsize}c@{\hskip2\cellsize}c@{\hskip2\cellsize}c@{\hskip2\cellsize}c}
 \tableau{ 3 & 3 & 3 \\ 2 } &   \tableau{ 3 & 3 & 3 \\ 1 } &  \tableau{ 3 & 3 & 2 \\ 1 } &  \tableau{ 3 & 2 & 2 \\ 1 } & \tableau{ 2 & 2 & 2 \\ 1 } 
 \end{array}
\end{displaymath}
\end{ex}

The \emph{monomial quasisymmetric polynomial} $M_\alpha(x_1,\ldots , x_n)$ is the generating function of what we call monomial reverse composition tableaux, which are those fundamental reverse composition tableaux in which all entries in the same row are equal. 

\begin{ex}
We have $M_{13}(x_1,x_2,x_3) = x^{013}+x^{103}+x^{130}$, as witnessed by the following monomial reverse composition tableaux.
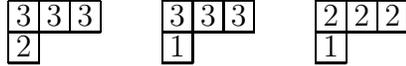
\begin{displaymath}
 \begin{array}{c@{\hskip2\cellsize}c@{\hskip2\cellsize}c@{\hskip2\cellsize}c@{\hskip2\cellsize}c}
 \tableau{ 3 & 3 & 3 \\ 2 } &   \tableau{ 3 & 3 & 3 \\ 1 } & \tableau{ 2 & 2 & 2 \\ 1 } 
 \end{array}
\end{displaymath}
\end{ex}

One may also define fundamental Young composition tableaux and monomial Young composition tableaux, by replacing the decreasing row condition with the corresponding increasing row condition in the definitions of fundamental (respectively, monomial) reverse composition tableaux. One could then define Young fundamental quasisymmetric polynomials and Young monomial quasisymmetric polynomials to be the generating functions of fundamental (respectively, monomial) Young composition tableaux. In this case, however, the polynomials remain the same.

\begin{proposition}\label{prop:YoungFisF}
The generating function of the fundamental Young composition tableaux of shape $\alpha$ is $F_\alpha(x_1, \ldots , x_n)$ and the generating function of the monomial Young composition tableaux of shape $\alpha$ is $M_\alpha(x_1, \ldots , x_n)$. 
\end{proposition}
\begin{proof}
By definition, the monomial reverse composition tableaux are exactly the monomial Young composition tableaux. Since every entry in any row of a fundamental reverse composition tableaux is strictly smaller than any entry in the row above, reversing the entries of every row is a weight-preserving bijection between fundamental reverse composition tableaux and fundamental Young composition tableaux of the same shape.
\end{proof}

We turn our attention to the quasisymmetric Schur polynomials $\qs_\alpha$ and the Young quasisymmetric Schur polynomials $\yqs_\alpha$, where we will see a distinction between the reverse and the Young models. To define quasisymmetric Schur polynomials, we first define \emph{triples} in reverse composition tableaux. These are collections of three boxes in $D(\alpha)$ with two adjacent in a row and either (Type A) the third box above the right box with the lower row weakly longer, or (Type B) the third box below the left box with the higher row strictly longer.  A triple of either type is said to be an \emph{inversion triple} if it is not the case that $z\ge y\ge x$. 

\begin{figure}[ht]
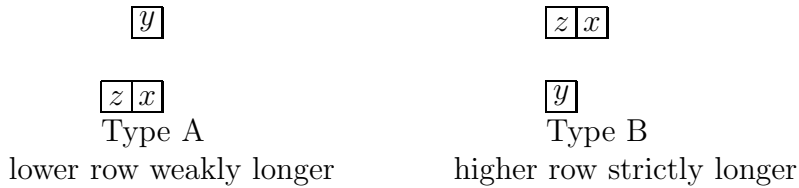

  \begin{displaymath}
    \begin{array}{l}
        \tableau{ & y } \\ \\ \tableau{  z & x } \\  \mbox{Type A} \\ \hspace{-3\cellsize} \mbox{lower row weakly longer}
    \end{array}
    \hspace{3\cellsize}
    \begin{array}{l}
   \hspace{3\cellsize}   \tableau{ z & x } \\ \\ \hspace{3\cellsize} \tableau{ y & } \\ \hspace{3\cellsize}\mbox{Type B} \\  \mbox{higher row strictly longer}
    \end{array}
  \end{displaymath}
  \caption{Triples for reverse composition tableaux.}\label{fig:reversetriples}
\end{figure}

Define the \emph{semistandard reverse composition tableaux} $\RCT(\alpha)$ for $\alpha$ to be the fillings of $D(\alpha)$ satisfying the following conditions.
\begin{enumerate}
\item Entries in each row weakly decrease from left to right.
\item Entries strictly increase from bottom to top in the first column.
\item All type A and type B triples are inversion triples.
\end{enumerate}
Then $\qs_\alpha(x_1, \ldots , x_n)$ is the generating function of $\RCT(\alpha)$ \cite{HLMvW09}.

\begin{ex}\label{ex:qs}
We have $\qs_{13}(x_1,x_2,x_3) = x^{013} + x^{022} + 2x^{112} + x^{103} + x^{202} + x^{121} + x^{211}+x^{130}+x^{220}$, as witnessed by the semistandard reverse composition tableaux in Figure~\ref{fig:QS13}.

\begin{figure}[ht]
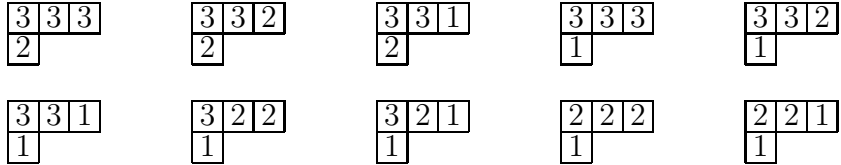

\begin{displaymath}
 \begin{array}{c@{\hskip3\cellsize}c@{\hskip3\cellsize}c@{\hskip3\cellsize}c@{\hskip3\cellsize}c@{\hskip3\cellsize}c@{\hskip3\cellsize}c}
 \tableau{ 3 & 3 & 3 \\ 2 } &  \tableau{ 3 & 3 & 2 \\  2 } &  \tableau{ 3 & 3 & 1 \\ 2 } &  \tableau{ 3 & 3 & 3 \\ 1 } &  \tableau{ 3 & 3 & 2 \\ 1 }   \\ \\  \tableau{ 3 & 3 & 1 \\ 1 } &  \tableau{ 3 & 2 & 2 \\ 1 } &   \tableau{ 3 & 2 & 1 \\ 1 }  &   \tableau{ 2 & 2 & 2 \\ 1 }  &   \tableau{ 2 & 2 & 1 \\ 1 } 
 \end{array}
\end{displaymath}
\caption{The ten elements of $\RCT(13)$ with entries at most $3$.}\label{fig:QS13} 
\end{figure}
\end{ex}

A \emph{Young triple} is a collection of three boxes with two adjacent in a row such that either (Type I) the third box is below the right box and the higher row is weakly longer, or (Type II) the third box is above the left box and the lower row is strictly longer (Figure~\ref{fig:Youngtriples}).  A Young triple of either type is said to be a \emph{Young inversion triple} if it is not the case that $x\ge y\ge z$.

\begin{figure}[ht]
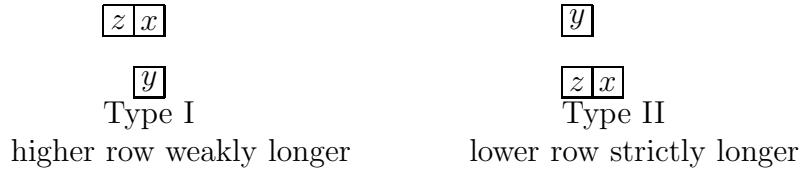

  \begin{displaymath}
    \begin{array}{l}
        \tableau{ z & x \\ & & \\ & y} \\  \mbox{Type I} \\ \hspace{-3\cellsize} \mbox{higher row weakly longer}
    \end{array}
    \hspace{3\cellsize}
    \begin{array}{l}
   \hspace{3\cellsize}  \tableau{ y \\ & & \\ z & x} \\ \hspace{3\cellsize}\mbox{Type II} \\  \mbox{lower row strictly longer}
    \end{array}
  \end{displaymath}
  \caption{Young triples for Young composition tableaux.}\label{fig:Youngtriples}
\end{figure}

Define the \emph{Young semistandard reverse composition tableaux} $\YCT(\alpha)$ for $\alpha$ to be the fillings of $D(\alpha)$ satisfying the following conditions.
\begin{enumerate}
\item Entries in each row weakly increase from left to right.
\item Entries strictly increase from bottom to top in the first column.
\item All type I and type II Young triples are Young inversion triples.
\end{enumerate}
Then the Young quasisymmetric Schur polynomial $\yqs_\alpha(x_1, \ldots , x_n)$ is the generating function of $\YCT(\alpha)$ \cite{LMvWbook}.

\begin{remark}
Young quasisymmetric Schur polynomials are most often defined in terms of a single triple condition; e.g \cite{LMvWbook}, \cite{AHM18}. While this is more compact, it does not extend appropriately to define a Young analogue of key polynomials.  The proof that these definitions are equivalent is analogous to the corresponding proof for reverse composition tableaux given in~\cite{HLMvW09}.
\end{remark}

\begin{ex}
We have $\yqs_{13}(x_1,x_2,x_3) = x^{130} + x^{121} + x^{112} + x^{103} + x^{013}$, as witnessed by the semistandard Young composition tableaux in Figure~\ref{fig:YQS103}. 

\begin{figure}[ht]
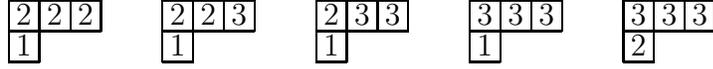

\begin{displaymath}
 \begin{array}{c@{\hskip2\cellsize}c@{\hskip2\cellsize}c@{\hskip2\cellsize}c@{\hskip2\cellsize}c@{\hskip2\cellsize}c@{\hskip2\cellsize}c}
   \tableau{ 2 & 2 & 2 \\ 1 } &     \tableau{ 2 & 2 & 3 \\ 1 } &    \tableau{ 2 & 3 & 3 \\ 1 } &    \tableau{ 3 & 3 & 3 \\ 1 } &    \tableau{ 3 & 3 & 3 \\ 2 }  
 \end{array}
\end{displaymath}
\caption{The five elements of $\YCT(13)$ with entries at most $3$.}\label{fig:YQS103} 
\end{figure}
\end{ex}

Notice that $\yqs_{13}(x_1,x_2,x_3) \neq \qs_{13}(x_1,x_2,x_3)$; indeed, they have a different number of terms. However, quasisymmetric Schur and Young quasisymmetric Schur polynomials are related by the following formula. 

\begin{proposition}\label{prop:yqstoqs}~\cite{LMvWbook}
Let $\alpha$ be a composition with at most $n$ parts. Then
\[\yqs_\alpha(x_1,\ldots , x_n) = \qs_{\rev(\alpha)}(x_n,\ldots , x_1).\]
\end{proposition}

\begin{remark}
As mentioned in the introduction, the \emph{flip-and-reverse} map on composition tableaux which reverses the order of the rows and exchanges entries $i \leftrightarrow (n+1-i)$ is a weight-reversing bijection between $\YCT(\alpha)$ and $\RCT(\rev(\alpha))$, implying Proposition~\ref{prop:yqstoqs}. In particular, reversing the order of the rows ensures the increasing first column condition is preserved.
\end{remark}

To illustrate this, we compute the Young quasisymmetric Schur polynomial $\yqs_{31}(x_1,x_2,x_3)$; compare this to the computation of $\qs_{13}(x_1,x_2,x_3)$ in Example~\ref{ex:qs}.

\begin{ex}
We have $\yqs_{23}(x_1,x_2,x_3) = x^{310} + x^{220} + 2x^{211} + x^{301} + x^{202} + x^{121} + x^{112} + x^{031}+x^{022}$, as witnessed by the semistandard Young composition tableaux in Figure~\ref{fig:YQS31}.
\begin{figure}[ht]
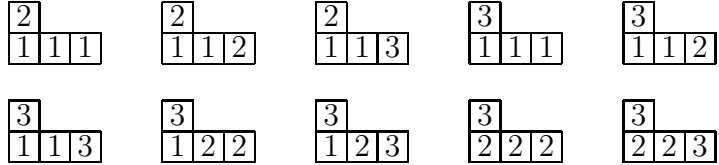

\begin{displaymath}
 \begin{array}{c@{\hskip2\cellsize}c@{\hskip2\cellsize}c@{\hskip2\cellsize}c@{\hskip2\cellsize}c@{\hskip2\cellsize}c@{\hskip2\cellsize}c}
 \tableau{ 2 \\ 1 & 1 & 1} &  \tableau{ 2 \\ 1 & 1 & 2  } &  \tableau{ 2 \\ 1 & 1 & 3 } &  \tableau{3 \\ 1 & 1 & 1 } &  \tableau{3 \\ 1 & 1 & 2 }   \\ \\  \tableau{3 \\ 1 & 1 & 3 } &  \tableau{ 3\\ 1 & 2 & 2 } &   \tableau{ 3\\ 1 & 2 & 3 }  &   \tableau{ 3 \\ 2 & 2 & 2 }  &   \tableau{ 3 \\ 2 & 2 & 3 } 
 \end{array}
\end{displaymath}
\caption{The ten elements of $\YCT(31)$ with entries at most $3$.}\label{fig:YQS31} 
\end{figure}
\end{ex}

Notice this involution preserves monomial and fundamental quasisymmetric polynomials: $M_{\alpha}(x_1,x_2, \hdots , x_n) = M_{\rev(\alpha)} (x_n, \hdots , x_2, x_1)$ and $F_{\alpha}(x_1,x_2, \hdots , x_n) = F_{\rev(\alpha)} (x_n, \hdots , x_2, x_1)$.

\begin{proposition}\label{prop:yqsqsexpansion}~\cite{HLMvW09,LMvWbook}
Quasisymmetric Schur and Young quasisymmetric Schur polynomials expand positively in the fundamental quasisymmetric basis, and
\[\yqs_\alpha(x_1,\ldots , x_n) = \sum_\beta c_\beta^\alpha F_\beta(x_1,\ldots , x_n)\]  
if and only if 
\[\qs_{\rev(\alpha)}(x_1,\ldots , x_n) = \sum_\beta c_\beta^\alpha F_{\rev(\beta)}(x_1,\ldots , x_n) .\]
\end{proposition}

For example, 
$\yqs_{31}(x_1,x_2,x_3) = F_{31}(x_1,x_2,x_3)+F_{22}(x_1,x_2,x_3)$, whereas 
$\qs_{13}(x_1,x_2,x_3) = F_{13}(x_1,x_2,x_3)+F_{22}(x_1,x_2,x_3).$

A remarkable property of the quasisymmetric Schur and Young quasisymmetric Schur polynomials is that they both positively refine Schur polynomials:

\begin{proposition}\label{prop:schurexpansion}\cite{LMvWbook}
\[s_\lambda(x_1,\ldots , x_n) = \sum_{\sort(\alpha) = \lambda}\qs_\alpha(x_1,\ldots , x_n) = \sum_{\sort(\alpha) = \lambda}\yqs_\alpha(x_1,\ldots , x_n)\]
\end{proposition} 

\begin{remark}
As noted in the introduction, Schur polynomials may be described in terms of either decreasing or increasing semistandard tableaux. Therefore Schur polynomials and ``Young Schur polynomials'' are the same (provided we consider a partition and its reversal to be the same), so from this perspective it makes sense that Schur polynomials expand positively into both the quasisymmetric Schur and Young quasisymmetric Schur bases. Similarly, the fact that both quasisymmetric Schur and Young quasisymmetric Schur polynomials expand positively in fundamental quasisymmetric polynomials (Proposition~\ref{prop:yqsqsexpansion}) makes sense due to the fact that fundamental quasisymmetric polynomials may also be described in terms of either increasing or decreasing tableaux (Proposition~\ref{prop:YoungFisF}), and thus are the same as ``Young fundamental quasisymmetric polynomials''.
\end{remark}

Typically a Young quasisymmetric Schur polynomial is not equal to any quasisymmetric Schur polynomial. However, we can classify their coincidences. We delay the proof to the appendix.

\begin{theorem}\label{thm:yqsqs}
$\yqs_\alpha(x_1, \ldots , x_n) = \qs_\beta(x_1, \ldots , x_n)$ if and only if $\alpha=\beta$ and either $\alpha$ has all parts the same, or all parts of $\alpha$ are $1$ or $2$, or $n=\ell(\alpha)$ and consecutive parts of $\alpha$ differ by at most $1$.
\end{theorem}

\subsection{Key polynomials and Demazure atoms}

We now shift our attention to the ring $\Poly_n=\mathbb{Z}[x_1,\ldots , x_n]$ of all polynomials in $n$ variables. This ring possesses a variety of bases important in geometry and representation theory. A principal example is the basis of \emph{key polynomials}, which are characters of (type A) Demazure modules \cite{Dem74a, LasSch90, RS95} and which also arise as specializations of nonsymmetric Macdonald polynomials.  Closely related is the basis of \emph{Demazure atoms}, originally introduced as \emph{standard bases} in \cite{LasSch90}.  Demazure atoms were shown in~\cite{Mas09} to also be a specialization of nonsymmetric Macdonald polynomials.  They are equal to the smallest non-intersecting pieces of type $A$ Demazure characters and can be obtained through a truncated application of \emph{divided difference operators}. Intuitively, one can build the Demazure atoms by starting with a monomial and partially symmetrizing, keeping only the monomials not appearing in the previous iteration of this process.

\subsubsection{Semi-skyline fillings}
Both key polynomials and Demazure atoms are defined in terms of reverse fillings that are often referred to as semi-skyline fillings. To define the key polynomial corresponding to a weak composition $a$ of length $n$, first note that the definition of type A and B triples extends verbatim from composition diagrams to weak composition diagrams. We need to include a \emph{basement column}, an extra $0$th column in the diagram: for our purposes the basement entry of row $i$ is $n+1-i$. Basement entries do not contribute to the weight of a filling. Define the \emph{key fillings} $\KSSF(a)$ for $a$ to be the fillings of $D(\rev(a))$ (note the reversal) satisfying the following conditions.
\begin{enumerate}
\item Entries in each row, including basement entries, weakly decrease from left to right.
\item Entries do not repeat in any column.
\item All type A and type B triples, including triples containing basement entries, are inversion triples.
\end{enumerate}
 We use the following as definitional for key polynomials.

\begin{theorem}\label{thm:keydefinition}\cite{HHL08,Mas09},
Let $a$ be a weak composition of length $n$. Then 
\[\key_a = \sum_{T\in \KSSF(a)}x^{\wt(T)},\]
where only the non-basement entries contribute to the weight.
\end{theorem}

For example, we have $\key_{032} = x^{032} + x^{122} + x^{212} + x^{302} + x^{311} + x^{320} + x^{131} + x^{221} + x^{230}$, 
which is computed using the elements of $\KSSF(032)$ shown in Figure~\ref{fig:key032} below.
\begin{figure}[ht]
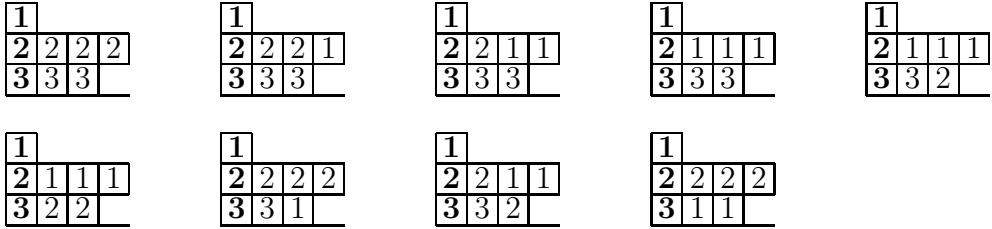

\begin{displaymath}
 \begin{array}{c@{\hskip3\cellsize}c@{\hskip3\cellsize}c@{\hskip3\cellsize}c@{\hskip3\cellsize}c@{\hskip3\cellsize}c@{\hskip3\cellsize}c}

  \tableau{  \bf{1}  \\   \bf{2} & 2 & 2 & 2  \\ \bf{3} & 3 & 3 \\ \hline }  &   \tableau{  \bf{1}  \\   \bf{2} & 2 & 2 & 1  \\ \bf{3} & 3 & 3 \\ \hline }  &    \tableau{  \bf{1}  \\   \bf{2} & 2 & 1 & 1  \\ \bf{3} & 3 & 3 \\ \hline }  &    \tableau{  \bf{1}  \\   \bf{2} & 1 & 1 & 1  \\ \bf{3} & 3 & 3 \\ \hline }  &    \tableau{  \bf{1}  \\   \bf{2} & 1 & 1 & 1  \\ \bf{3} & 3 & 2 \\ \hline }  \\ \\    \tableau{  \bf{1}  \\   \bf{2} & 1 & 1 & 1  \\ \bf{3} & 2 & 2 \\ \hline }  &    \tableau{  \bf{1}  \\   \bf{2} & 2 & 2 & 2  \\ \bf{3} & 3 & 1 \\ \hline }  &    \tableau{  \bf{1}  \\   \bf{2} & 2 & 1 & 1  \\ \bf{3} & 3 & 2 \\ \hline }  &    \tableau{  \bf{1}  \\   \bf{2} & 2 & 2 & 2  \\ \bf{3} & 1 & 1 \\ \hline }     
   \end{array}
\end{displaymath}
\caption{The 9 key fillings of shape $032$. (Basement entries in bold.)}\label{fig:key032} 
\end{figure}

The definition of the Demazure atoms in terms of semi-skyline fillings comes from specializing the diagram fillings used to generate the nonsymmetric Macdonald polynomials~\cite{HHL08}. Define the \emph{atom fillings} $\ASSF(a)$ for $a$ to be the fillings of $D(a)$ (no basement) satisfying the following conditions.
\begin{enumerate}
\item Entries weakly decrease from left to right in each row.
\item Entries do not repeat in any column.
\item The first entry of each row is equal to its row index. 
\item All type A and type B triples are inversion triples.
\end{enumerate}
We use the following as definitional for Demazure atoms.

\begin{theorem}\label{thm:atomdefinition}\cite{Mas09}
Let $a$ be a weak composition of length $n$. Then 
\[\atom_a = \sum_{T\in \ASSF(a)}x^{\wt(T)}.\]
\end{theorem}

\subsubsection{Left and right keys}

The eponymous formula for the key polynomial $\key_a$ is given in terms of right keys. A \emph{semistandard Young tableau} (or $\SSYT$) $T$ is a tableau of partition shape such that entries weakly increase along rows and strictly increase up columns. For a partition $\lambda$, let $\SSYT(\lambda)$ denote the set of all $\SSYT$ of shape $\lambda$, and $\SSYT_n(\lambda)$ the subset of $\SSYT(\lambda)$ whose entries are at most $n$. 
 A semistandard Young tableau $T$ is a \emph{key} if the entries appearing in the $(i+1)$th column of $T$ are a subset of the entries appearing in the $i^{th}$ column of $T$, for all $i$. For a weak composition $a$, define $\keytab(a)$ to be the unique key of weight $a$. For any semistandard Young tableau $T$, there are two keys of the same shape as $T$ associated to $T$, called the \emph{right key} of $T$, denoted $K_+(T)$, and the \emph{left key} of $T$, denoted $K_-(T)$
 
We now describe procedures for computing right and left keys, which will be illustrated in Example~\ref{ex:leftrightkey} below. There are several different methods for computing keys (see, for example~\cite{Mas09},~\cite{Wil13}) but we use the classical method presented in \cite{RS95} as it involves several tools we will need later.  
Two words ${\bf b}$ and ${\bf c}$ in $\{1,2,\ldots n\}$ are said to be \emph{Knuth-equivalent}, written ${\bf b}\sim {\bf c}$, if one can be obtained from the other by a series of the following local moves:
\begin{align*}
{\bf d}xzy{\bf e} \sim {\bf d}zxy{\bf e} & \quad \mbox{ for } \quad x\le y < z \\ 
{\bf d}yxz{\bf e} \sim {\bf d}yzx{\bf e} & \quad \mbox{ for } \quad  x< y \le z 
\end{align*}
for words ${\bf d}$ and ${\bf e}$ and letters $x,y,z$. 

Define the \emph{column word factorization} of a word $v$ to be the decomposition of $v$ into subwords $v=v^{(1)} v^{(2)} \cdots$ by starting a new subword between every weak ascent.  Then the \emph{column form} of $v$ (denoted $\colform(v)$) is the composition whose parts are the lengths of the subwords appearing in the column word factorization.  Let $\lambda$ be the shape of the $\SSYT$ obtained when Schensted insertion (see, e.g.,~\cite{Ful97,Sag13,Sta99}) is applied to $v$.  The word $v$ is said to be \emph{column-frank} if $\colform(v)$ is a rearrangement of the nonzero parts of $\lambda'$, where $\lambda'$ denotes the conjugate shape of $\lambda$ obtained by reflecting the diagram of $\lambda$ across the line $y=x$.  Let $T\in \SSYT(\lambda)$.  Then the right key (resp. left key) of $T$ is the key of shape $\lambda$ whose $j^{th}$ column is equal to the last (resp. first) subword in any column-frank word which is Knuth equivalent to the \emph{column word} $\col(T)$ of $T$ (obtained by reading the entries of $T$ down columns from left to right) and whose last (resp. first) subword has length $\lambda_j'$.

Notice the difference in the construction of left and right keys.  The weight of the left key is usually not a reversal of the weight of the right key; the subtle connection between left and right keys is explored in Section~\ref{sec:leftkeys}, wherein we also define polynomials naturally associated to left keys.

\begin{theorem}\label{thm:rightkey}~\cite{LasSch90,RS95}
Let $a$ be a weak composition of length $n$. Then
$$\key_{a} = \sum_{\substack{ T \in \SSYT_n(\sort(a)) \\ K_+(T) \le \keytab(a)}} x^{\wt(T)} ,$$

where $K_+(T)\le \keytab(a)$ if each entry of $K_+(T)$ is weakly smaller than the corresponding entry of $\keytab(a)$.
\end{theorem}

\begin{ex}\label{ex:leftrightkey}
Let $a=032$. Then $\keytab(a) =   \tableau{  3 & 3  \\  2 & 2 & 2 }$, which is a tableau of shape $\lambda = 32$. The nine tableaux whose right keys are smaller than or equal to $\keytab(a)$ are 
\[
\begin{array}{c@{\hskip1.5\cellsize}c@{\hskip1.5\cellsize}c@{\hskip1.5\cellsize}c@{\hskip1.5\cellsize}c@{\hskip1.5\cellsize}c@{\hskip1.5\cellsize}c@{\hskip1.5\cellsize}c@{\hskip1.5\cellsize}c@{\hskip1.5\cellsize}c@{\hskip1.5\cellsize}c}

   \tableau{  3 & 3  \\  2 & 2 & 2 }  &  \tableau{  3 & 3  \\  1 & 2 & 2 } &   \tableau{  3 & 3  \\  1 & 1 & 2 }  &   \tableau{  3 & 3  \\  1 & 1 & 1 }  &    \tableau{  2 & 3  \\  1 & 1 & 1 } &  
  
   \tableau{  2 & 2  \\  1 & 1 & 1 } &    \tableau{  2 & 2  \\  1 & 1 & 2 } &  \tableau{  2 & 3  \\ 1 & 1 & 2 }  &  \tableau{  2 & 3  \\  1 & 2 & 2 } 
  \end{array}
\]
To illustrate the process of finding right (and left) keys, let $T$ be the last tableau in the list above. Then $\col(T) = 21322$. The words that are Knuth-equivalent to $21322$ (listed with vertical bars indicating the column word factorizations) are $\{21|32|2, \; 21|2|32, \; 2|21|32, \; 2|2|31|2\}$. The column form of the first three words is a rearrangement of $221$, the shape of $\lambda'$, so these three words are column-frank. The fourth is not column-frank so we ignore it. Looking at the rightmost subword in each column-frank word, the first of these words tells us that the column of $K_+(T)$ of length $1$ consists of a single $2$, and the second (or third) word tells us that the columns of $K_+(T)$ of length $2$ each contain a $2$ and a $3$. 

Thus $K_+(T) = \tableau{  3 & 3  \\  2 & 2 & 2 }$. Similarly, via leftmost subwords, we obtain  $K_-(T) = \tableau{  2 & 2  \\  1 & 1 & 2 }$.
\end{ex}

One may also use right keys to define the Demazure atoms.  Given a weak composition $a$ of length $n$, the Demazure atom $\atom_a$ can also be given by \begin{align}\label{atomrtkey} \atom_{a} = \sum_{\substack{ T \in \SSYT_n(\lambda(a)) \\ K_+(T) = \keytab(a)}} x^{\wt(T)}. \end{align}
 
From this construction and Theorem~\ref{thm:rightkey}, it is apparent that key polynomials expand positively in Demazure atoms.  In particular, 
\begin{align}\label{keysintoatoms} \key_a = \sum_{b \le a} \atom_b,\end{align} 
where $b \le a$ if and only if $\sort(b)=\sort(a)$ and the permutation $w$ such that $w(\sort(b))=b$ is less than or equal to the permutation $v$ such that $v(\sort(a))=a$ in the Bruhat order. 

\subsubsection{Divided differences and crystal graphs}{\label{sec:keycrystal}}

Key polynomials can be defined in terms of \emph{divided difference operators}. Given a positive integer $i$, where $1\le i <n$, define an operator $\partial_i$ on $\mathbb{Z}[x_1,\ldots , x_n]$ by 
\[\partial_i(f) =  \frac{f-s_i(f)}{x_i-x_{i+1}}\]
where $s_i$ exchanges $x_i$ and $x_{i+1}$. Now define another operator $\pi_i$ on $\mathbb{Z}[x_1,\ldots , x_n]$ by 
\[\pi_i(f) = \partial_i(x_if).\]
For a permutation $w$, define $\pi_w = \pi_{i_1}\cdots \pi_{i_r}$, where $s_{i_1}\cdots s_{i_r}$ is any reduced word for $w$. (This definition is independent of the choice of reduced word because the $\pi_i$ satisfy the commutation and braid relations for the symmetric group.)  Recall that $\sort(a)$ is the rearrangement of the entries of $a$ into decreasing order.  For a weak composition $a$ let $w_a$ be the minimal length permutation that sends $a$ to $\sort(a)$ acting on the right.  Then the key polynomial is given by
\[\key_a = \pi_{w_a}x^{\sort(a)}.\]

\begin{ex}
Let $a=032$. Then the minimal length permutation taking $a$ to $\sort(a) = 320$ is $s_1 s_2$. We compute
\begin{align*}
\pi_1\pi_2 (x_1^3x_2^2) & =\pi_1 \frac{x_1^3x_2^3 - x_1^3x_3^3}{x_2-x_3} \\
                                              & = \pi_1(x_1^3x_2^2+x_1^3x_2x_3+x_1^3x_3^2) \\
                                              & = \frac{(x_1^4x_2^2 - x_1^2x_2^4) + (x_1^4x_2x_3 - x_1x_2^4x_3) +(x_1^4x_3^2-x_2^4x_3^2)}{x_1-x_2} \\
                                              & = x_1^3x_2^2+x_1^2x_2^3+x_1^3x_2x_3+x_1^2x_2^2x_3+x_1x_2^3x_3+x_1^3x_3^2+x_1^2x_2x_3^2+x_1x_2^2x_3^2+x_2^3x_3^2 \\
                                              & = \key_{032}.                                             
\end{align*}
\end{ex}

Demazure atoms can also be described in terms of divided difference operators.  In particular, let $\overline{\pi_i} = \pi_i -1$.  Then (see ~\cite{Mas09}) $$\atom_a=\overline{\pi}_{w_a} x^{\sort(a)}.$$

The action of the divided difference operators can be realised in terms of \emph{Demazure crystals}.  A \emph{crystal graph} is a directed and colored graph whose edges are defined by \emph{Kashiwara operators}~\cite{Kas91,Kas93,Kas95} $e_i$ and $f_i$. See~\cite{HonKan02} for a detailed introduction to the theory of quantum groups and crystal bases and~\cite{BS17} for a more combinatorial exploration of crystals.  

For a partition $\lambda$, the type $A_n$ highest weight crystal $B(\lambda)$ of highest weight $\lambda$ has vertices indexed by $\SSYT_n(\lambda)$. The \emph{character} of $B(\lambda)$ is 
\[\ch(B(\lambda)) = \sum_{T\in B(\lambda)}x^{\wt(T)},\]
which is equal to the Schur polynomial $s_{\lambda}(x_1, \ldots , x_n)$, reflecting the fact that Schur polynomials are characters for irreducible highest weight modules for $GL_n$.  See Figure~\ref{fig:Youngcrystalkey} below for $B(21)$ when $n=3$, in which the arrows index the Kashiwara operators $f_1$ and $f_2$. Precise definitions of the $f_i$ can be found in e.g.~\cite{BS17}; in particular we note that $f_i(b)=0$ if there is no $i$-arrow emanating from vertex $b$, and the $e_i$ are defined by $e_i(b)=b'$ if $f_i(b')=b$, and $e_i(b)=0$ otherwise.

\begin{figure}[ht]
\begin{center}
\begin{tikzpicture}[xscale=1.5,yscale=1.2]
  \node at (2,4) (T112) {$\tableau{2 \\ 1 & 1}$};
  \node at (0,3) (T113) {$\tableau{3 \\ 1 & 1}$};
  \node at (4,3) (T122) {$\tableau{2 \\ 1 & 2}$};
  \node at (1,2) (T132) {$\tableau{2 \\ 1 & 3}$};
  \node at (3,2) (T123) {$\tableau{3 \\ 1 & 2}$};
  \node at (0,1) (T133) {$\tableau{3 \\ 1 & 3}$};
  \node at (4,1) (T223) {$\tableau{3 \\ 2 & 2}$};
  \node at (2,0) (T233) {$\tableau{3 \\ 2 & 3}$};

  \draw[thick,->,blue  ] (T112) -- (T122) node[midway,above] {$1$};
  \draw[thick,->,blue  ] (T113) -- (T123) node[midway,above] {$1$};
  \draw[thick,->,blue  ] (T123) -- (T223) node[midway,above] {$1$};
 \draw[thick,->,blue  ] (T133) -- (T233) node[midway,above] {$1$};

  \draw[thick,->,red  ] (T112) -- (T113) node[midway,above] {$2$};
  \draw[thick,->,red  ] (T122) -- (T132) node[midway,above] {$2$};
  \draw[thick,->,red  ] (T132) -- (T133) node[midway,above] {$2$};
  \draw[thick,->,red  ] (T223) -- (T233) node[midway,above] {$2$};
\end{tikzpicture}
\caption{\label{fig:Youngcrystalkey} Crystal graph $B(21)$ for $n=3$.}
\end{center}
\end{figure}
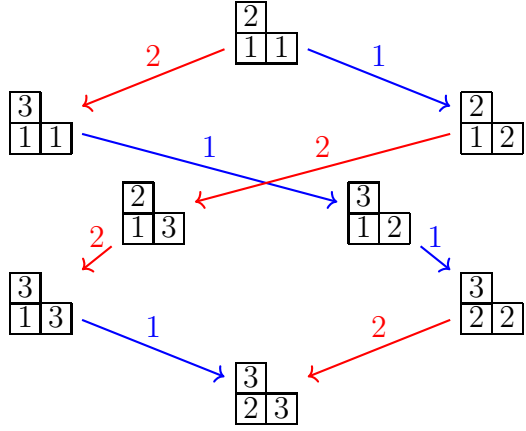 

A Demazure crystal is a subset of $B(\lambda)$ whose character is a key polynomial~\cite{Lit95, Kas93}, obtained by a truncated action of the Kashiwara operators. Specifically, given a subset $X$ of $B(\lambda)$, define operators $\mathfrak{D}_i$ for $1 \le i < n$ by 
$$\mathfrak{D}_i X = \{ b \in B(\lambda) | e_i^r(b) \in X  \textrm{ for some } r \ge 0 \}.$$  
Given a permutation $w$ with reduced word $w = s_{i_1} s_{i_2} \cdots s_{i_k}$, define 
$$B_w(\lambda) = \mathfrak{D}_{i_1} \mathfrak{D}_{i_2} \cdots \mathfrak{D}_{i_k} \{ u_\lambda \},$$ 
where $u_\lambda$ is the highest weight element in $B(\lambda)$, i.e., $e_i(u_{\lambda}) = 0$ for all $1 \le i < n$.  If $b,b' \in B_w(\lambda) \subseteq B(\lambda)$ and $f_i(b)=b'$ in $B(\lambda)$, then the crystal operator $f_i$ is also defined in $B_w(\lambda)$.  The character of a Demazure crystal $B_w(\lambda)$ is defined as $$\ch B_w (\lambda) = \sum_{b \in B_w(\lambda)} x_1^{\wt(b)_1} \cdots x_n^{\wt(b)_n},$$ which is equal to $\key_a$ when $w$ is of shortest length such that $w(a)=\lambda$~\cite{Lit95,Kas93}. The repeated actions of the $\mathfrak{D}_i$ starting with $u^\lambda$ precisely mirrors the repeated action of the divided difference operators $\pi_i$ starting with the monomial $x^\lambda$.

\begin{ex}\label{ex:crystalkey}
Let $a=102$. Then the shortest length $w$ such that $w(a)=\sort(a) = 210$ is $w=s_2s_1$. Therefore, the crystal graph for $\key_{102}$ is the subgraph of $B(21)$ consisting of all vertices that can be obtained from the highest weight $\tableau{2 \\ 1 & 1}$ by first applying a sequence of $f_1$'s and then a sequence of $f_2$'s. In Figure~\ref{fig:Youngcrystalkey}, these are the tableaux of weight $210$, $201$, $120$, $111$ (the leftmost such) and $102$. Hence $\key_{102} = x_1^2 x_2 + x_1x_2^2 + x_1^2 x_3 + x_1 x_2 x_3 + x_1 x_3^2$.  
\end{ex}

\subsubsection{Compatible Sequences}{\label{sec:CompatibleSequences}}

Key polynomials can also be constructed using \emph{compatible sequences} as follows.  Let ${\bf b} = b_1 b_2 \cdots b_p$ be a word in the alphabet $\{1,2,\ldots n\}$.  A word ${\bf w}=w_1 w_2 \cdots w_p$ is \emph{${\bf b}$-compatible} if
\begin{enumerate}
\item $1 \le w_1 \le w_2 \le \cdots \le w_p\le n$,
\item $w_k < w_{k+1}$ whenever $b_k < b_{k+1},$ for all $1 \le k <p$, and
\item $w_k \le b_k$ for all $1 \le k \le p$ (flag condition).
\end{enumerate}

\begin{theorem}\cite{RS95}\label{thm:keycompatible}
Let $a$ be a weak composition of length $n$. Then 
$$\key_a = \sum_{\rev(b) \sim \col(\keytab(a)), \; w \textrm{ is $b$-compatible}} x^{\comp(w)},$$
where $\comp(w)$ is the weak composition whose $i^{th}$ entry counts the incidences of $i$ in $w$. 
\end{theorem}

\begin{ex}\label{ex:keycompatible}
Let $a=032$. We have ${\rm key}(032)=\tableau{3 & 3 \\ 2 & 2 & 2 }$, and $\col({\rm key}(032)) = 32322$. The set of words Knuth-equivalent to $32322$ is $\{32322, 33222, 32232, 23232, 23322\}$. Reversing these gives the set 
$\{22323, 22233, 23223, 23232, 22332\}.$ We compute the set of compatible sequences for each of these:

\begin{figure}[h]
\begin{tabular}{l | l}
{\bf Word} & {\bf Compatible sequences} \\\hline
22323 & 11223 \\\hline
22233 & 22233 12233 11233 11133 11123 11122 \\\hline
23223 & 12223 \\\hline
23232 & \\\hline
22332 & 11222 \\
\end{tabular}
\caption{Compatible sequences. \label{fig:compseq}}
\end{figure}

\end{ex}

Observe there are $9$ compatible sequences, each having the weight $\comp(w)$ of a monomial of $\key_{032}$. In Proposition~\ref{prop:keytoslidecompatible}, we interpret the \emph{fundamental slide} expansion of a key polynomial in terms of Knuth equivalence classes.

\section{Young key polynomials}{\label{Sec:Youngkeys}}

We now introduce the \emph{Young key polynomial} basis for polynomials. This basis has proved useful in computing the Hilbert series of a generalization of the coinvariant algebra, specifically, in constructing a Gr\"{o}bner basis for the ideal $I_{n,k} = \langle x_1^k ,x_2^k, \hdots , x_n^k , e_n, e_{n-1} , \hdots , e_{n-k+1} \rangle$ \cite{HRS18}. However, the combinatorial and representation-theoretic properties of the Young key polynomials have not, to our knowledge, been explored previously, nor has the connection to the overall flip-and-reverse perspective. We begin by providing a combinatorial description of the Young key polynomial basis analogous to that of the Young version of the quasisymmetric Schur polynomials. 

Note that the definition of Young triples extends verbatim to weak composition diagrams. As in the definition of key polynomials, we append a \emph{basement column} to diagrams. Given a weak composition $a$ of length $n$, define the \emph{Young key fillings} $\YKSSF(a)$ for $a$ to be the fillings of $D(\rev(a))$ (note the reversal) with entries from $\{1, \ldots , n\}$ satisfying the following conditions.
\begin{enumerate}
\item Entries in each row, including basement entries, weakly increase from left to right.
\item Entries do not repeat in any column.
\item All type I and type II Young triples, including triples using basement entries, are Young inversion triples.
\end{enumerate}
Define the \emph{Young key polynomial} $\ykey_a$ by
\[\ykey_a = \sum_{T\in \YKSSF(a)}x^{\wt(T)},\]
where only the non-basement entries contribute to the weight. 

For example, we have $\ykey_{230} = x^{230} + x^{221} + x^{212} + x^{203} + x^{113} + x^{023} + x^{131} + x^{122} + x^{032}$, which is computed by the elements of $\YKSSF(230)$ shown in Figure~\ref{fig:ykey230}.

\begin{figure}[ht]
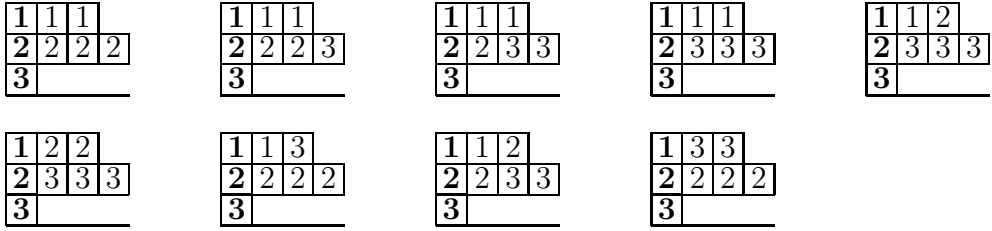

\begin{displaymath}
 \begin{array}{c@{\hskip3\cellsize}c@{\hskip3\cellsize}c@{\hskip3\cellsize}c@{\hskip3\cellsize}c@{\hskip3\cellsize}c@{\hskip3\cellsize}c}

\tableau{ \bf{1} & 1 & 1 \\ \bf{2} & 2 & 2 & 2 \\  \bf{3} \\ \hline } &  \tableau{ \bf{1} & 1 & 1 \\ \bf{2} & 2 & 2 & 3 \\  \bf{3} \\ \hline } &   \tableau{ \bf{1} & 1 & 1 \\ \bf{2} & 2 & 3 & 3 \\  \bf{3} \\ \hline } &   \tableau{ \bf{1} & 1 & 1 \\ \bf{2} & 3 & 3 & 3 \\  \bf{3} \\ \hline } &   \tableau{ \bf{1} & 1 & 2 \\ \bf{2} & 3 & 3 & 3 \\  \bf{3} \\ \hline } \\ \\   \tableau{ \bf{1} & 2 & 2 \\ \bf{2} & 3 & 3 & 3 \\  \bf{3} \\ \hline } &   \tableau{ \bf{1} & 1 & 3 \\ \bf{2} & 2 & 2 & 2 \\  \bf{3} \\ \hline } &   \tableau{ \bf{1} & 1 & 2 \\ \bf{2} & 2 & 3 & 3 \\  \bf{3} \\ \hline } &   \tableau{ \bf{1} & 3 & 3 \\ \bf{2} & 2 & 2 & 2 \\  \bf{3} \\ \hline }    
\end{array}
\end{displaymath}
\caption{The 9 Young key fillings of shape $230$. (Basement entries are in bold.)}\label{fig:ykey230} 
\end{figure}

Note that the definition immediately implies that 
\begin{equation}\label{eqn:keyyoungkey}
\ykey_a(x_1, x_2 , \hdots , x_n) = \key_{\rev(a)}(x_n, x_{n-1} , \hdots , x_1).
\end{equation}

\begin{proposition}\label{prop:youngkeybasis}
The Young key polynomials are a basis for $\Poly_n$, containing the Schur polynomials. In particular, if $a$ is decreasing then
\[\ykey_a = s_\lambda(x_1,\ldots , x_n),\]
where $\lambda$ is $a$ with trailing zeros removed.
\end{proposition}
\begin{proof}
The Young key polynomials are equinumerous with the key polynomials. Any polynomial can be expressed as a linear combination of key polynomials (since key polynomials are a basis of $\Poly_n$), and thus as a linear combination of Young key polynomials by (\ref{eqn:keyyoungkey}). Hence the Young key polynomials are a basis of $\Poly_n$.

We have $\key_{\rev(a)} = s_{a}$ \cite{Mac91}, hence $\ykey_a = s_a = s_\lambda$ by (\ref{eqn:keyyoungkey}) and because Schur polynomials are symmetric, hence invariant under exchanging variables.
\end{proof}

In this way, both the key and Young key polynomials extend the Schur polynomials to $\Poly_n$. This is in fact their only coincidence.

\begin{theorem}\label{thm:keyyoungkeyintersect}
The polynomials that are both key polynomials and Young key polynomials are exactly the Schur polynomials. 
\end{theorem}
\begin{proof}
Suppose $s_{\lambda}(x_1,\ldots , x_n)$ is a Schur polynomial in $n$ variables. Then 
\[s_{\lambda}(x_1,\ldots , x_n) = \key_{0^{n-\ell(\lambda)}\times \rev(\lambda)} = \ykey_{\lambda\times 0^{n-\ell(\lambda)}},\]
where $0^m\times b$ (respectively, $b\times 0^m$) denotes $b$ with $m$ zeros prepended (respectively, appended).

For the converse, note that for any weak composition $a$, the key polynomial $\key_a$ has the monomial $x^{\sort(a)}$ as a term; this follows from the divided difference definition. But the only Young key polynomial containing $x^{\sort(a)}$ as a term is $\ykey_{\sort(a)}$ itself, which is a Schur polynomial. So if $\key_a$ is not a Schur polynomial it cannot be equal to any Young key polynomial.
\end{proof}

We also define a Young analogue of the Demazure atoms. Let $a$ be a weak composition of length $n$. Define the \emph{Young atom fillings} $\YASSF(a)$ for $a$ to be the fillings of $D(a)$ (no basement) with entries from $\{1, \ldots , n\}$ satisfying the following conditions.
\begin{enumerate}
\item Entries weakly increase from left to right in each row.
\item Entries do not repeat in any column.
\item All type I and type II Young triples are Young inversion triples.
\item The first entry of each row is equal to its row index.
\end{enumerate}
Define the \emph{Young atom} $\ya_a$ by
\[\ya_a = \sum_{T\in \YASSF(a)}x^{\wt(T)}.\]

The definition immediately implies that $\ya_a(x_1, x_2 , \hdots , x_n) = \atom_{\rev(a)}(x_n, x_{n-1} , \hdots , x_1).$  Similar to Proposition~\ref{prop:youngkeybasis}, the Young atoms form a basis of $\Poly_n$.  We can establish the coincidences between Demazure atoms and Young atoms, as we did in Theorem~\ref{thm:keyyoungkeyintersect} for keys and Young keys. Note the condition for coincidence is less restrictive than that for coincidence of quasisymmetric Schur and Young quasisymmetric Schur polynomials (Theorem~\ref{thm:yqsqs}), due to elements of $\YASSF(a)$ and $\ASSF(a)$ necessarily having identical first column. 

\begin{theorem}\label{thm:atomyatom}
The polynomials that are both Demazure atoms and Young atoms are precisely the $\ya_a$ such that $|a_i-a_{i+1}|\le 1$ for all $1\le i < n$. 
\end{theorem} 
\begin{proof}
First we show that if $\ya_a = \atom_b$ then $a=b$. Suppose $\max(a) > \max(b)$, where $\max(a)$ is the largest part of $a$. Then since entries cannot repeat in any column for either $\YASSF$ or $\ASSF$, $\ya_a$ has terms where some $x_i$ has degree $\max(a)$, but $\atom_b$ cannot have any such term.  Hence if $\ya_a=\atom_b$, the longest row(s) in $D(a)$ and $D(b)$ must have the same length. By a similar argument, the next-longest rows must then have the same length, etc. 
Thus if $\ya_a = \atom_b$, then $b$ must be a rearrangement of $a$. 

Now suppose $b$ rearranges $a$. Let $T\in \YASSF(a)$ be such that all entries in the $j$th row (for each $j$) are equal to $j$, and suppose there exists $S\in \ASSF(b)$ with the same weight as $T$. By definition, the first entry in each row $j$ of $S$ is $j$. Because the rows of $b$ rearrange those of $a$, the number of boxes in each column of $D(b)$ is the same as that for each column of $D(a)$. It follows that the set of entries in each column of $S$ must be the same as that in the corresponding column of $T$, since $T$ has $a_j$ instances of each entry $j$, and entries cannot repeat in any column of $T$ or $S$.  

Now consider the entries in the second column of $S$, which are a subset of the entries in the first column for both $S$ and $T$. None of these entries can go in a row above the row that contains that entry in the first column, else the two copies of that entry must violate one of the triple conditions. Nor can they go in a row below, since entries must decrease along each row. So each entry must go immediately adjacent to the same entry in the first column of $S$. Continuing thus, we obtain $S=T$, so in particular $a=b$.

Now suppose $a_{i} - a_{i+1} \ge 2$ for some $i$. Let $T\in \YASSF(a)$ be such that all entries in each row $j$ are $j$, and let $T'$ be obtained by changing the rightmost $i$ in $T$ to $i+1$. Since $a_{i} - a_{i+1} \ge 2$, this new $i+1$ is not in the first column, and is at least two columns to the right of any other $i+1$, so no $\YASSF$ properties are affected by this change and $T'\in \YASSF(a)$. But there is no $S\in \ASSF(a)$ with weight equal to $T'$: in rows $i+1$ and above, entries in $S$ must agree with entries in $T'$, and then there is nowhere the new $i+1$ could be placed in $S$. Hence $\ya_a\neq \atom_a$. A similar argument shows that if $a_{i+1} - a_{i} \ge 2$, then $\atom_a\neq \ya_a$.

Conversely, it is straightforward to observe that if $|a_i-a_{i+1}|\le 1$ for all $1\le i < n$, then both $\ya_a$ and $\atom_a$ are equal to the single monomial $x^a$.
\end{proof}

\subsection{Compatible sequences}{\label{sec:compatibleseqs}
The Young key polynomials may also be described in terms of compatible sequences. 
For a word $w$ in $\{1,2,\ldots , n\}$ define the \emph{flip} of $w$ to be the word $f(w)$ in $\{1,2,\ldots , n\}$ obtained by replacing each entry $w_i$ with $n+1-w_i$. Also define the \emph{flip-reverse} of $w$, denoted $\frev(w)$, to be the word $f(\rev(w))$, or equivalently $\rev(f(w))$.
\begin{ex}
If $n=6$ and $w=2446154$, then $f(w)=5331623$ and $\frev(w)=3261335$. 
\end{ex}

Let $T$ be an $\SSYT$. Define the \emph{right-to-left column reading word} $\col_R(T)$ to be the word obtained by reading the entries in each column of $T$ from top to bottom starting with the rightmost column and moving from right to left.  

\begin{lemma}\label{lem:frevrighttoleft}
Let $a$ be a weak composition. Then $\frev(\col(\keytab(a)))=\col_R(\keytab(\rev(a)))$.
\end{lemma}
\begin{proof}
First of all, $\keytab(a)$ and $\keytab(\rev(a))$ have the same shape.  To see this, note that the height of the $i^{th}$ column of $\keytab(a)$ is equal to the number of entries $a_j$ in $a$ such that $a_j \ge i$.  This number is the same for $a$ and $\rev(a)$.

This also shows that for any given column, the entries of that column in $\keytab(a)$ are the flips of the entries of that column in $\keytab(\rev(a))$. Hence when the word for $\keytab(\rev(a))$ is reversed, the column breaks line up and the word in each column is the flip-reverse of the word in that column of $\keytab(a)$. The statement follows.
\end{proof}

\begin{ex}
Let $a=(2,4,0,3)$. We have 
\[\keytab(a) =  \tableau{ 4 & 4 \\ 2 & 2 & 4 \\ 1 & 1 & 2 & 2   } \qquad \mbox{ and } \qquad \keytab(\rev(a))= \tableau{  4 & 4 \\ 3 & 3 & 3  \\ 1 & 1 & 1 & 3 }.\]
Here $\col(\keytab(a))$ is $421|421|42|2$ and $\col_R(\keytab(\rev(a)))$ is $3|31|431|431$, which is indeed equal to $\frev(\col(\keytab(a)))$ (column-breaks included for emphasis).
\end{ex}

The following lemma is fairly well-known~\cite[Appendix A.1]{Ful97}; we include a proof here for completeness and to illustrate the flip-and-reverse procedure.

\begin{lemma}\label{lem:knuthfrev}
Let $w,w'$ be words in $\{1,\ldots , n\}$. Then $w\sim w'$ if and only if $\frev(w)\sim \frev(w')$. 
\end{lemma}
\begin{proof}
It is enough to show this for the case that $w$ and $w'$ are related by a single Knuth move. For $x$ a letter in $w$, let $\overline{x}$ denote $n+1-x$. Suppose $w$ contains the sequence $\ldots xzy \ldots$ where $x\le y < z$. Then one may perform a Knuth move to obtain $w' = \ldots zxy \ldots$. In $\frev(w)$ we have $\ldots \overline{y}\overline{z}\overline{x} \ldots$ where $\overline{z}<\overline{y}\le \overline{x}$. Then one may perform a Knuth move to obtain the word $\ldots \overline{y}\overline{x}\overline{z} \ldots$, which is indeed $\frev(w')$. Now suppose $w$ contains the sequence $\ldots yxz \ldots$ where $x< y \le z$. Then one may perform a Knuth move to obtain $w' = \ldots yzx \ldots$. In $\frev(w)$ we have $\ldots \overline{x}\overline{z}\overline{y} \ldots$ where $\overline{z}\le\overline{y}< \overline{x}$. Then one may perform a Knuth move to obtain the word $\ldots \overline{z}\overline{x}\overline{y} \ldots$, which is indeed $\frev(w')$.

Therefore, $\frev(w)\sim \frev(w')$ whenever $w\sim w'$. The converse direction is immediate from the fact that $\frev$ is an involution.
\end{proof}

\begin{proposition}\label{prop:Knutheq}
Let $a$ be a weak composition. Then $\col(\keytab(a))$ is Knuth-equivalent to $\col_R(\keytab(a))$.
\end{proposition}

\begin{proof}
It suffices to show that the word $\col_R(\keytab(a))$ inserts to $\keytab(a)$.

Suppose $T$ is a key. Let $w$ be a word that contains the entries in the leftmost column of $T$ and let $T'$ be the key obtained by removing the leftmost column of $T$.  We will show that inserting the entries of $w$ into $T'$ in order from largest to smallest yields another key, namely $T'$ with the column whose entries are the entries of $w$ adjoined on the left.  This is the key $T$ and the conclusion then follows by induction, the base case where $T$ is empty being trivial.

We will establish that insertion of the $i$th entry of $w$ causes (a copy of) the $(i-1)$th entry of $w$ to be bumped from the first into the second row, the $(i-2)$nd entry of $w$ to be bumped from the 2nd to the 3rd row, etc, culminating in the first entry of $w$ arriving at the end of the $i$th row. This is clearly true for $i=1$, as the largest entry of $w$ is weakly larger than any entry of $T$ (due to the key condition) so it is inserted at the end of the first row. Suppose this is true for all entries up to the $(i-1)$th entry of $w$. Now, when the $i$th entry of $w$ is inserted, it bumps (a copy of) the $(i-1)$th entry of $w$ from row $1$, since there is no entry $x$ in the tableau such that $w_i<x<w_{i-1}$ by the key condition. Then the $(i-1)$th entry of $w$ must bump (a copy of) the $(i-2)$nd entry of $w$ (which is in row $2$ by the inductive hypothesis), since again there is no entry $y$ in the tableau such that $w_{i-1}<y<w_{i-2}$ by the key condition. Continuing thus, $w_1$ is eventually bumped into row $i$, and comes to rest at the end of row $i$ since it is weakly larger than any other entry in the tableau. 

Hence the insertion process results in a new entry $w_i$ in each row $|w|+1-i$. There is a unique such semistandard Young tableau, and by the key condition each entry $w_i$ (or a copy of this entry) must appear as the first entry of row $|w|+1-i$ for every $i$. Therefore the result is $T'$ with the column determined by $w$ appended, as required.
\end{proof}

We now give a formula for Young key polynomials in terms of compatible sequences.

\begin{theorem}\label{thm:ykeycompatible}
Let $a$ be a weak composition of length $n$. Then
$$\ykey_a = \sum_{f(c) \sim \col(\keytab(a)), \, w \textrm{ is $c$-compatible}} x^{\comp(f(w))}.$$
\end{theorem}
\begin{proof}
The set $X$ of words Knuth-equivalent to $\col(\keytab(\rev(a)))$ is equal to the set of words Knuth-equivalent to $\col_R(\keytab(\rev(a)))$ by Proposition~\ref{prop:Knutheq}, which is equal to the set of words Knuth-equivalent to $\frev(\col(\keytab(a)))$ by Lemma~\ref{lem:frevrighttoleft}. Then by Lemma~\ref{lem:knuthfrev}, the flip-reverses of the words in $X$ form the set $Y$ of words Knuth-equivalent to $\col(\keytab(a))$. Since $Y = \{\frev(x) : x\in X\}$, we have $\{f(y):y\in Y\} = \{\rev(x): x\in X\}$. By Theorem~\ref{thm:keycompatible}, $\key_{\rev(a)}(x_1,\ldots , x_n)$ is generated by the compatible sequences for $\{\rev(x): x\in X\}$, and thus also generated by the compatible sequences for $\{f(y):y\in Y\}$. Since $\ykey_a(x_n, \ldots , x_1) = \key_{\rev(a)}(x_1, \ldots , x_n)$, the compatible sequences for $\{f(y):y\in Y\}$ generate $\ykey_a(x_n, \ldots , x_1)$, i.e., 
$$\ykey_a(x_n, \ldots , x_1) = \sum_{f(c) \sim \col(\keytab(a)), \, w \textrm{ is $c$-compatible}} x^{\comp(w)}.$$ 
Finally, flipping each compatible sequence in the formula above yields $\ykey_a(x_1, \ldots , x_n)$.
\end{proof}

\begin{ex}\label{ex:ykeyrevcompatible}
Let $a=230$. Then $\keytab(a) = \tableau{  2  & 2 \\ 1 & 1 & 2 }$; its column word is $21212$. The set of words Knuth-equivalent to $21212$ is $\{22121, 22211, 21221, 21212, 22112\}$. 
We compute the set of compatible sequences for the flips of each of these words.

\begin{figure}[h]
\begin{tabular}{l | l | l | l}
{\bf Word} & {\bf flip} & {\bf Compatible sequences} & {\bf Flips of Compatible sequences} \\\hline
22121 & 22323 & 11223 & 33221 \\\hline
22211 & 22233 & 11122 \, 11123 \, 11133 & 33322 \, 33321 \, 33311 \\
 & &   11233 \, 12233 \, 22233 &   33211 \, 32211 \, 22211 \\\hline
21221 & 23223 & 12223  & 32221  \\\hline
21212 & 23232 & & \\\hline
22112 & 22332 & 11222 & 33222  \\
\end{tabular}
\caption{Compatible sequences and their flips \label{fig:revcompseq}}
\end{figure}


\end{ex}

The corresponding monomials indeed sum up to $\ykey_{230}$; compare this example to Example~\ref{ex:keycompatible} computing $\key_{032}$ in terms of compatible sequences.

\subsection{Divided differences and Demazure crystals}
Young key polynomials may also be described in terms of divided difference operators. Given a weak composition $a$, let $\revsort(a)$ be the rearrangement of $a$ into increasing order. Let $\hat{w}_a$ be the permutation of shortest length rearranging $a$ to $\revsort(a)$. For $1 \le i < n$ define an operator
\[\hat{\pi}_i = -\partial_i x_{i+1},\]
and for a permutation $w$, define $\hat{\pi}_w = \hat{\pi}_1\cdots \hat{\pi}_r$, where $s_1\cdots s_r$ is any reduced word for $w$.

\begin{lemma}\label{lem:reversediff}
Let $f$ be a polynomial in $\mathbb{Z}[x_1,\ldots , x_n]$. We have
\[I({\pi}_i f) = \hat{\pi}_{n-i} I(f)\]
where $I(f)$ is the polynomial obtained by exchanging variables $x_j\leftrightarrow x_{n+1-j}$.
\end{lemma}
\begin{proof}
By linearity, it suffices to show this is true for a monomial $f=x^b$, where $b$ is a weak composition of length $n$. We compute 
\begin{align*}
I(\pi_ix^b) & = I\left(\frac{x_1^{b_1} \cdots  x_i^{b_i+1} x_{i+1}^{b_{i+1}} \cdots x_n^{b_n} - x_1^{b_1} \cdots  x_i^{b_{i+1}} x_{i+1}^{b_{i}+1} \cdots x_n^{b_n}}{x_i-x_{i+1}}\right) \\ 
                & = \frac{x_n^{b_1} \cdots  x_{n+1-i}^{b_i+1} x_{n-i}^{b_{i+1}} \cdots x_1^{b_n} - x_n^{b_1} \cdots  x_{n+1-i}^{b_{i+1}} x_{n-i}^{b_{i}+1} \cdots x_1^{b_n}}{x_{n+1-i}-x_{n-i}}
\end{align*}
and
\begin{align*}
\hat{\pi}_{n-i} I (x^b) & = \hat{\pi}_{n-i} (x_1^{b_n} \cdots  x_{n-i}^{b_{i+1}} x_{n+1-i}^{b_i} \cdots x_n^{b_1}) \\ 
                                     & = \frac{x_1^{b_n} \cdots  x_{n-i}^{b_{i+1}} x_{n+1-i}^{b_i+1} \cdots x_n^{b_1} -x_1^{b_n} \cdots  x_{n-i}^{b_i+1} x_{n+1-i}^{b_{i+1}}  \cdots x_n^{b_1}}{x_{n+1-i}-x_{n-i}}
\end{align*}
as required.
\end{proof}

\begin{lemma}
$\hat{\pi}_w$ is well-defined.
\end{lemma}
\begin{proof}
Since the $\pi_i$'s satisfy the commutativity and braid relations of $S_n$, it follows from Lemma~\ref{lem:reversediff} that the $\hat{\pi}_i$'s also do.
\end{proof}

\begin{theorem}\label{thm:Youngkeydivideddifference}
Let $a$ be a weak composition of length $n$. Then
$\ykey_a = \hat{\pi}_{\hat{w}_a}x^{\revsort(a)}.$
\end{theorem}
\begin{proof}
First observe that if $w_a=s_{i_1}\cdots s_{i_k}$ is the minimal length permutation sending $a$ to $\sort(a)$, then $s_{n-i_1}\cdots s_{n-i_k}$ is the minimal length permutation sending $\rev(a)$ to $\revsort(\rev(a))$, i.e., is $\hat{w}_{\rev(a)}$.

Therefore, by Lemma~\ref{lem:reversediff} and the fact that $I(x^{\sort(a)}) = x^{\revsort(a)}=x^{\revsort(\rev(a))}$, we have
\[\hat{\pi}_{\hat{w}_{\rev(a)}}x^{\revsort(\rev(a))} = \hat{\pi}_{{\hat{w}_{\rev(a)}}}I(x^{\sort(a)}) = I(\pi_{w_a}(x^{\sort(a)})) = I(\key_a) = \ykey_{\rev(a)}.\qedhere \]
\end{proof}

\begin{ex}
Let $a=230$. Then the minimal length permutation taking $a$ to $\revsort(a) = 023$ is $s_2 s_1$. We compute
\begin{align*}
\hat{\pi}_2\hat{\pi}_1 (x_2^2x_3^3) & = \hat{\pi}_2\frac{x_2^3x_3^3 - x_1^3x_3^3}{x_2-x_1} \\
                                              & = \hat{\pi}_2(x_2x_3^3+x_1x_2^2x_3^3 + x_1^3x_3^3) \\
                                              & = \frac{(x_2^2 x_3^4 - x_2^4x_3^2) + (x_1 x_2 x_3^4 - x_1 x_2^4 x_3) + (x_1^2 x_3^4 - x_1^2 x_2^4)}{x_3-x_2} \\
                                              & = x_2^2 x_3^3 + x_2^3 x_3^2 + x_1 x_2 x_3^3  + x_1 x_2^2 x_3^2 + x_1 x_2^3 x_3 + x_1^2 x_3^3 + x_1^2 x_2 x_3^2 + x_1^2 x_2^2 x_3 + x_1^2 x_2^3 \\
                                              & = \ykey_{230}.                                             
\end{align*}
\end{ex}

Recall the Demazure crystal structure for key polynomials described in Section~\ref{sec:keycrystal}. The Young key polynomials may be realized as characters of crystals that are obtained via Demazure truncations beginning from the \emph{lowest} weight of the crystal $B(\lambda)$ rather than the highest. For a subset $X$ of $B(\lambda)$, define $\hat{\mathfrak{D}}_iX = \{ b \in B(\lambda) | f_i^r(b) \in X  \textrm{ for some } r \ge 0 \}.$

\begin{theorem}
Let $a$ be a weak composition of length $n$ and let $w$ be of shortest length such that $w(a) = \revsort(a)$. Then the Young key polynomial $\ykey_a$ is the character of the subcrystal of $B(\sort(a))$ obtained by 
\[\hat{\mathfrak{D}}_{i_1}\cdots \hat{\mathfrak{D}}_{i_k}\{\hat{u}_\lambda\},\]
where $s_{i_1}\cdots s_{i_k}$ is a reduced word for $w$ and $\hat{u}_\lambda$ is the lowest weight element of $B(\lambda)$.
\end{theorem}
\begin{proof}
Recall that the shortest permutation sending $\rev(a)$ to $\sort(\rev(a))$ is $s_{n-i_1}\cdots s_{n-i_k}$.  Performing the \emph{Lusztig involution} \cite{Lus10} $\star$ on $B(\lambda)$ exchanges each $f_i$ with $e_{n-i}$ and $e_i$ with $f_{n-i}$, and reverses the weight of each vertex. Hence, applying a Demazure truncation with $s_{n-i_1}\cdots s_{n-i_k}$ from the highest weight of $B(\lambda)^\star$ yields $\key_{\rev(a)}$ with variables reversed, which is equal to $\ykey_a$ by (\ref{eqn:keyyoungkey}). The statement follows.
\end{proof}

Observe that the repeated actions of the $\hat{\mathfrak{D}}_i$ starting with $\hat{u}_\lambda$ precisely mirrors the repeated action of the divided difference operators $\hat{\pi}_i$ starting with the monomial $x^{0^{n-\ell(\lambda)}\times \rev(\lambda)}$.

\begin{ex}\label{ex:youngcrystalkey}
Let $a=201$, and recall $B(21)$ from Figure~\ref{fig:Youngcrystalkey}. The shortest-length $w$ such that $w(a)=\revsort(a)$ is $w=s_1s_2$. Therefore, the crystal graph for $\ykey_{201}$ is the subgraph of $B(21)$ consisting of all vertices that can be obtained from the lowest weight $\tableau{3 \\ 2 & 3}$ by first applying a sequence of $e_2$'s and then a sequence of $e_1$'s. Hence $\ykey_{201} = x_2 x_3^2 + x_2^2 x_3 + x_1 x_3^2 + x_1 x_2 x_3 + x_1^2 x_3$.
\end{ex}

\subsection{Young key polynomials as generators for left keys}{\label{sec:leftkeys}}
Recall Theorem~\ref{thm:rightkey} states that a key polynomial can be described as the generating function for the set of all $\SSYT$ with bounded right key.  In this section we provide an analogous description of Young key polynomials as well as the corresponding description of Young Demazure atoms.

Given a semistandard Young tableau $T$, let $\frev(T)$ denote the filling obtained by flipping all entries in $T$ and reversing the order of the resulting column entries.  Compare this to the definition of $\frev$ applied to a word at the beginning of Section~\ref{sec:compatibleseqs}.  It is a straightforward observation that when $T$ is a key, $\frev(T)$ is the key whose entries in each column are the flip-reverses of the entries in the corresponding column of $T$. (However, if $T$ is not a key then $\frev(T)$ is not necessarily even a semistandard Young tableau.)

We need to establish a weight-reversing bijection between semistandard Young tableaux with a given right key $U$ and semistandard Young tableaux with left key $\frev(U)$. This is done in the following lemma, which can also be understood in terms of the \emph{evacuation} operation on semistandard Young tableaux. Recall that a word $w$ is Knuth equivalent to a semistandard Young tableau $T$ if and only if Schensted insertion of the word $w$ produces the tableau $T$.  

\begin{lemma}\label{lem:rightkeyleftkey}
Let $T$ be a semistandard Young tableau. Then the left key of the tableau obtained via Schensted insertion of $\frev(\col(T))$ is $\frev(K_+(T))$.
\end{lemma}

\begin{proof}
Let $T$ have shape $\lambda$ and let $U$ be the semistandard Young tableau obtained by Schensted insertion of $\frev(\col(T))$.  Consider any column index $j$. Consider any word $w'$ that is Knuth equivalent to $\col(T)$, has column form a rearrangement of $\lambda'$, and whose rightmost maximal decreasing subsequence has length $\lambda_j'$. Then the entries in column $j$ of $K_+(T)$ are the entries of the rightmost maximal decreasing subsequence of $w'$. Now, the column form of $\frev(w')$ is the reversal of the column form of $w'$ (thus also a rearrangement of $\lambda'$), and Lemma~\ref{lem:knuthfrev} implies that $\frev(\col(T))$ is Knuth equivalent to $\frev(w')$.   Therefore the leftmost maximal decreasing subsequence of $\frev(w')$ is the flip-reverse of the rightmost maximal decreasing subsequence of $w'$, and hence the entries in the the $j^{th}$ column of the left key of $U$ are precisely the flip-reverses of the entries in the $j^{th}$ column of the right key of $T$.
\end{proof}

\begin{theorem}{\label{thm:leftkeygen}}
The Young Demazure atoms and Young key polynomials are generated by the left keys of semistandard Young tableaux as follows:

$$\ya_a= \sum_{\substack{T \in \SSYT_n( \lambda (a)) \\ K_-(T) = \keytab(a)}} x^{\wt(T)} \qquad \mbox{ and }\qquad \ykey_{a} = \sum_{\substack{T \in \SSYT_n( \lambda (a)) \\ K_-(T) \ge \keytab(a)}} x^{\wt(T)},$$ 
where $\ge$ means entrywise comparison and $n=\ell(a)$.
\end{theorem}

\begin{proof}
Consider the first expansion. Recall that $\ya_a(x_1,\ldots , x_n) = \atom_{\rev(a)}(x_n,\ldots , x_1)$ and that (by Equation~\ref{atomrtkey}) $\atom_{\rev(a)}$ is generated by the set of all semistandard Young tableaux whose right key equals $\keytab(\rev(a))$. It is therefore enough to exhibit a weight-reversing bijection between the set of all semistandard Young tableaux whose right key equals $\keytab(\rev(a))$ and the set of all semistandard Young tableaux whose left key is $\keytab(a)$.

We know from Lemma~\ref{lem:rightkeyleftkey} that if $T$ is a semistandard Young tableau such that $K_+(T)=\keytab(\rev(a))$, then the semistandard Young tableau $S$ obtained via Schensted insertion of $\frev(\col(T))$ has $K_-(S) = \frev(K_+(T)) = \frev(\keytab(\rev(a))) = \keytab(a)$. This process is clearly invertible, hence bijective, and the application of $\frev$ to $\col(T)$ ensures it is weight-reversing.

For the second expansion, we recall that $\ykey_a(x_1,\ldots , x_n) = \key_{\rev(a)}(x_n, \ldots , x_1)$ and that by Theorem~\ref{thm:rightkey} $\key_{\rev(a)}$ is generated by the set of all semistandard Young tableaux whose right key is less than or equal to $\keytab(\rev(a))$. It is straightforward to check that if $K_+(T)\le \keytab(\rev(a))$, then the semistandard Young tableau $S$ obtained via Schensted insertion of $\frev(\col(T))$ has $K_-(S) \ge \frev(K_+(T)) = \keytab(a)$. The second expansion then follows by applying the same argument used to prove the first expansion.
\end{proof}

\begin{ex}{\label{ex:frev}}
Let $T = \tableau{ 3 & 4 \\1 & 1 & 2 }$, which has right key $K_+(T) = \tableau{ 4 & 4 \\ 2 & 2 & 2 }$. We have $\col(T) =  3  1  4  1  2$. Schensted insertion of $\frev(\col(T))=3  4  1  4  2$ produces the semistandard Young tableau $\tableau{3 & 4 \\ 1 & 2 & 4}$ which indeed has left key $\tableau{3 & 3 \\ 1 & 1 & 3} = \frev(K_+(T))$.
\end{ex}

\subsection{Row-frank words}
Our next aim is to realize Young key polynomials as traces on modules. For this, we first adapt a formula of \cite{LasSch90} expressing key polynomials in terms of \emph{row-frank} words.  The first condition below is equivalent to the condition of being row-frank; see~\cite{RS95} for details.  The \emph{standardization} of a semistandard Young tableau $T$, denoted $\std(T)$, is the standard Young tableau obtained by replacing the $1$'s in $T$ from left to right by $1,2, \hdots , \gamma_1$, the $2$'s by $\gamma_1 +1 , \gamma_1 +2, \hdots , \gamma_1 + \gamma_2$, and so on, where $\gamma_i$ equals the number of times the entry $i$ appears in $T$. Given a word $u$ in positive integers, its \emph{row-word factorization} is $\cdots u^{(2)} u^{(1)}$, where each \emph{row-word} $u^{(i)}$ is a weakly increasing subsequence of maximal length. 

For a weak composition $a$, let $\w(a)$ be the set of all words $u = \cdots u^{(2)} u^{(1)}$ with each $u^{(i)}$ having $a_i$ letters, satisfying the following conditions.
\begin{enumerate}
\item The word $u$ maps to a pair $(P,\std(\keytab(a)))$ under the \emph{column insertion} described in \cite{RS95}.
\item No letter of $u^{(i)}$ exceeds $i$.
\end{enumerate}

\begin{theorem}{\label{thm:frankkey}}\cite{LasSch90}
The key polynomials are generated using words in $\w(a)$ as follows:
$$\key_a = \sum_{u \in \w(a)} x_u.$$ 
\end{theorem}

We now provide the analogue of this generating function for Young key polynomials. 
For a weak composition $a$, let $\yw(a)$ be the set of all words $u = \cdots u^{(2)} u^{(1)}$ with each $u^{(i)}$ having $a_i$ letters, satisfying the following conditions.
\begin{enumerate}
\item The word $\frev(u)$ maps to a pair $(P,\std(\keytab(\rev(a))))$ under column insertion.
\item For each letter $j$ of $u^{(i)}$, we have $i \le j \le \ell(a)$.
\end{enumerate}

\begin{ex}\label{ex:WandYW}
We have 
\[\w(032)=\{33| 222 |, 33| 122 |, 33| 112 |, 33| 111|, 23| 111|, 23| 112|, 23| 122|, 22|111|, 22|112| \}\]
and
\[\yw(230)=\{| 222| 11, | 223| 11, | 233| 11, | 333| 11, | 333| 12, | 233 | 12, | 223 | 12, |333|22, |233|22 \},\]
 where the vertical bars denote the row word factorization (including empty row-words).
\end{ex}

\begin{theorem}{\label{thm:frankyoungkey}}
The Young key polynomials are generated using the words in $\yw(a)$ as follows:
$$\ykey_a = \sum_{w \in \yw(a)} x_w.$$ 
\end{theorem}

\begin{proof}
Consider a word $u$ in $\w(a)$ and let $w=\frev(u)$.  Then $w$ satisfies condition (1) for $\yw(a)$ by construction.  Consider a letter $b$ in $u^{(i)}$.  By definition, $b \le i$. The flip $n-b+1$ of $b$ appears in the $(n-i+1)^{th}$ row-word of $w$, and $b \le i$ implies $n-b+1 \ge n-i+1$.  So $w$ satisfies both the conditions to be in the set $\yw(a)$. Since flipping and reversing is an invertible process, we have that the words in $\yw(a)$ are exactly the flip-reverses of the words in $\w(\rev(a))$. Then since the monomials appearing in $\ykey_a(x_1,\ldots x_n)$ are the flips of those appearing in $\key_{\rev(a)}(x_n,\ldots, x_1)$, it follows from Theorem~\ref{thm:frankkey} that $\yw(a)$ generates $\ykey_a$.
\end{proof}


\subsection{Young key polynomials as traces on modules} 

In \cite{RS95}, \emph{generalized flagged Schur modules} and \emph{key modules} are defined. The key polynomials are realized as traces on key modules, which are a special case of generalized flagged Schur modules. In this section we modify the Reiner-Shimozono approach to construct modules so that the Young key polynomials are realised as traces on these modules. 

As in~\cite{RS95}, a \emph{diagram} $D$ is a finite subset of the Cartesian product $\mathbb{P} \times \mathbb{P}$ of the positive integers with itself, where every element of $\mathbb{P} \times \mathbb{P}$ in $D$ is thought of as being a box.  A \emph{filling of shape $D$} is a map $T: D \rightarrow \mathbb{P}$ assigning a positive integer to each box in $D$ (note this is called a \emph{tableau of shape $D$} in \cite{RS95}).  

Let $\mathbb{F}$ be a field of characteristic $0$, and let $\mathcal{T}^{n}_D$ be the vector space over $\mathbb{F}$ with basis the set of all fillings $T$ of shape $D$ whose largest entry does not exceed $n$. Fix an order $\mathfrak{b}_1, \mathfrak{b}_2, \ldots$ on the boxes of $D$, and identify the filling $T$ with the tensor product $\epsilon_{T(\mathfrak{b}_1)}\otimes \epsilon_{T(\mathfrak{b}_2)}\otimes \cdots$, where $\epsilon_i$ is the $i$th standard basis vector. Then an action of $GL_n(\mathbb{F})$ on  $\mathcal{T}^{n}_D$ is defined by letting $GL_n(\mathbb{F})$ act on each $\epsilon_i$ as usual and extending this action linearly.  

The \emph{row group} $R(D)$ (respectively \emph{column group} $C(D)$) is the set of all permutations of the boxes of $D$ which fixes the rows (resp. columns) in which the boxes appear.  These groups act on $\mathcal{T}^{n}_D$ by permuting the positions of the entries within a filling.  As in~\cite{RS95}, define 
\[e_T = \sum_{\alpha \in R(D), \,\, \beta \in C(D)} {\rm sgn}(\beta) T\alpha \beta,\] 
where $T\alpha \beta$ is the filling obtained by acting first by $\alpha$ and then by $\beta$.

Define the \emph{Young generalized flagged Schur module $\yflagschur^{n}_{D}$} for an arbitrary diagram $D$ (with $n$ at least the maximum row index of $D$) to be the subspace of $\mathcal{T}^{n}_D$ spanned by the set $\{ e_T \}$ as $T$ runs over all fillings of shape $D$ whose entries in row $i$ are not smaller than $i$. It is straightforward that $\yflagschur^{n}_{D}$ is a $B$-module, where $B$ is the Borel subgroup of $GL_n(\mathbb{F})$ consisting of lower-triangular matrices.

\begin{remark}\label{rmk:youngvsreverse}
The construction of the \emph{generalized flagged Schur module $\flagschur_{D}$} described in~\cite{RS95} is similar, but serves to illustrate an important difference in the behaviors of Young and reverse families of polynomials. In \cite{RS95} $\mathcal{T}_D$ is defined to be the vector space with basis consisting of all fillings of shape $D$, with no restriction on the size of the entries. In this way, $\mathcal{T}_D$ is a $GL_\infty(\mathbb{F})$-module. Then $\flagschur_{D}$ is spanned by the set $\{ e_T \}$ as $T$ runs over all fillings of shape $D$ whose entries in row $i$ are not \emph{larger} than $i$, which is finite even though $\mathcal{T}_D$ is infinite-dimensional. In this way, $\flagschur_{D}$ is a module for the opposite Borel subgroup $B_-$ consisting of upper-triangular elements of $GL_\infty(\mathbb{F})$. The dependence on $n$ in the Young case is reflected in the fact that appending zeros to a weak composition does not change the corresponding key polynomial, but does change the Young key polynomial.
\end{remark}

\begin{ex}
Let $a=032$. Then if $T \mbox{ $=$ } \vline \tableau{ 2 & 3 \\ 1 & 2 & 2 \\ \\ \hline}$, applying elements of the row group to $T$ yields the following:
\begin{displaymath}
 \begin{array}{c@{\hskip1.5\cellsize}c@{\hskip1\cellsize}c@{\hskip1.5\cellsize}c@{\hskip1.5\cellsize}c@{\hskip1.5\cellsize}c@{\hskip1.5\cellsize}c}
2 \,\,\vline \tableau{ 2 & 3 \\  1 & 2 & 2  \\ \\ \hline} &  2  \,\,\vline \tableau{  2 & 3 \\  2 & 1 & 2 \\ \\ \hline }  & 2  \,\,\vline \tableau{  2 & 3 \\ 2 & 2 & 1 \\ \\ \hline }  &   2  \,\,\vline \tableau{  3 & 2 \\ 1 & 2 & 2 \\ \\ \hline }  &    2 \,\, \vline \tableau{  3 & 2 \\  2 & 1 & 2 \\ \\ \hline }  & 2 \,\, \vline \tableau{  3 & 2 \\ 2 & 2 & 1 \\ \\ \hline }  
    \end{array}
\end{displaymath}
    where the coefficients are $2$ because there are two distinct permutations yielding each ordering of $1,2,2$. It is easy to see that for any filling $S$ with repeated entries in any column, we have $\sum_{ \beta \in C(D)} {\rm sgn}(\beta) S \beta = 0$, hence only the first and fifth fillings above contribute to $e_T$. Applying the column group to each of these and summing the resulting fillings yields
\[
e_T  =  2 \,\, \vline \tableau{ 2 & 3 \\  1 & 2 & 2  \\ \\ \hline}  - 2  \,\, \vline \tableau{  1 & 3 \\  2 & 2 & 2 \\ \\ \hline }   - 2  \,\, \vline \tableau{  2 & 2 \\ 1 & 3 & 2 \\ \\ \hline }   + 2  \,\, \vline \tableau{  1 & 2 \\ 2 & 3 & 2 \\ \\ \hline }  + 2  \,\, \vline \tableau{  3 & 2 \\  2 & 1 & 2 \\ \\ \hline }   - 2  \,\,\vline \tableau{  2 & 2 \\ 3 & 1 & 2 \\ \\ \hline }   - 2 \,\, \vline \tableau{  3 & 1 \\ 2 & 2 & 2 \\ \\ \hline }   + 2  \,\, \vline \tableau{  2 & 1 \\ 3 & 2 & 2 \\ \\ \hline }
\]
\end{ex}

Define the \emph{key module} $\keymod_a$ for the weak composition $a$ to be the $B_-$-module $\flagschur_{D(a)}$.

\begin{theorem}{\cite{RS95}}{\label{thm:keymod}}
For $u = \cdots u^{(2)} u^{(1)}$ in $\w(a),$ let $T(u)$ be the filling of shape $D(a)$ obtained by placing $u^{(j)}$ in row $j$.  Then $\{ e_{T(u)} : u \in \w(a)\}$ is a basis for the key module $\keymod_a$.
\end{theorem}

We now describe the variation on the Reiner-Shimozono construction that needed to describe the Young key polynomials as characters. 
Let $a$ be a weak composition of length $n$, and define the \emph{Young key module} $\ykeymod_a$ for the weak composition $a$ to be the $B$-module $\yflagschur_{D(a)}$. Here we may drop $n$ from the notation, since $n$ is determined by the weak composition $a$.

\begin{cor}
For $u = u^{(1)} u^{(2)} \cdots$ in $\yw(a)$, let $T(u)$ be the filling of shape $D(a)$ obtained by placing $u^{(j)}$ in row $j$.  Then $\{ e_{T(u)} : u \in \yw(a)\}$ is a basis for the Young key module $\ykeymod_a$.
\end{cor}

\begin{proof}
The flip-and-reverse map on fillings extends linearly to an involution $\psi$, hence an isomorphism, on  $\mathcal{T}^{n}_{D(a)}$. Moreover, $\psi$ sends a filling whose entries are at least their row index to a filling whose entries are at most their row index, and vice versa. In particular, by the proof of Theorem~\ref{thm:frankyoungkey}, $\psi$ carries $\{ e_{T(u)} : u \in \yw(a)\}$ to the basis $\{ e_{T(\frev(u))} : \frev(u) \in \w(\rev(a))\}$ of $\keymod_{\rev(a)}$ given by Theorem~\ref{thm:keymod}. 

Therefore, $\{ e_{T(u)} : u \in \yw(a)\}$ is a linearly independent set, since any linear dependence in this set would imply, via $\psi$, a linear dependence in the linearly independent set $\{ e_{T(\frev(u))} : \frev(u) \in \w(\rev(a))\}$. Similarly $\{ e_{T(u)} : u \in \yw(a)\}$ is spanning: suppose $e_T\in \ykeymod_a$. Then $\psi(e_T)\in \keymod_{\rev(a)}$, hence is in the span of the spanning set $\{ e_{T(\frev(u))} : \frev(u) \in \w(\rev(a))\}$ of $\keymod_{\rev(a)}$, and applying $\psi$ again yields $e_T$ as a linear combination of $\{ e_{T(u)} : u \in \yw(a)\}$.
\end{proof}

\begin{remark}
The order in which entries of $u^{(j)}$ are placed in row $j$ does not matter, since fillings with any given ordering of $u^{(j)}$ in each row $j$ appear in $e_T$ due to the action of the row group. In Example~\ref{ex:ykeymod}, we represent $e_T$ by the filling $T$ with entries increasing from left to right in each row, which agrees with the choices of representatives for key modules in \cite{RS95}.
\end{remark}

Let $x$ be the diagonal matrix whose diagonal entries are $x_1, x_2, \ldots , x_n$. We immediately obtain the following (compare to Corollary 14 in~\cite{RS95}).

\begin{cor}
The Young key polynomial $\ykey_a$ is the trace of $x$ acting on the Young key module $\ykeymod_a$.
\end{cor}

\begin{ex}\label{ex:ykeymod}
The Young key module $\ykeymod_{230}$ has basis $\{ e_T \}$ for the following fillings $T$. 

\begin{displaymath}
 \begin{array}{c@{\hskip2\cellsize}c@{\hskip2\cellsize}c@{\hskip2\cellsize}c@{\hskip2\cellsize}c@{\hskip2\cellsize}c@{\hskip2\cellsize}c@{\hskip2\cellsize}c@{\hskip2\cellsize}c@{\hskip2\cellsize}c@{\hskip2\cellsize}c}
\vline \tableau{ \\ 2 & 2 & 2 \\  1 & 1  \\ \hline} &   \vline \tableau{ \\ 2 & 2 & 3 \\  1 & 1  \\ \hline} &\vline \tableau{ \\ 2 & 3 & 3 \\  1 & 1  \\ \hline} &\vline \tableau{ \\ 3 & 3 & 3 \\  1 & 1  \\ \hline} &\vline \tableau{ \\ 2 & 3 & 3 \\  1 & 2  \\ \hline}  & 
\vline \tableau{ \\ 2 & 2 & 3 \\  1 & 2  \\ \hline} &\vline \tableau{ \\ 3 & 3 & 3 \\  2 & 2  \\ \hline} &\vline \tableau{ \\ 2 & 3 & 3 \\  2 & 2  \\ \hline}.
    \end{array}
\end{displaymath}

\end{ex}

\section{Other polynomial families and intersections}{\label{Sec:others}}
In this section, we provide a new formula in terms of Knuth equivalence for the \emph{fundamental slide} expansion of a key polynomial, and interpret compatible sequences in terms of the \emph{fundamental particle} basis, introduced in \cite{Sea20} as a common refinement of the fundamental slide and Demazure atom bases.  As we did for Young key polynomials and Young atoms, we also determine the intersections of further reverse bases and their Young analogues. 

\subsection{The fundamental and monomial slide bases}

For a weak composition $a$, define the \emph{fundamental fillings} $\FF(a)$ for $a$ \cite{Sea20} to be the (reverse) fillings of $D(a)$ satisfying the following conditions.
\begin{enumerate}
\item Entries weakly decrease from left to right in each row.
\item No entry in row $i$ is greater than $i$.
\item If a box with label $b$ is in a lower row than a box with label $c$, then $b<c$.
\end{enumerate}
The \emph{fundamental slide polynomial} $\fs_a$ \cite{Assaf.Searles} is the generating function of $\FF(a)$:
\[\fs_a = \sum_{T\in \FF(a)}x^{\wt(T)}.\]
For example, $\fs_{103} = x^{103}+x^{112}+x^{121}+x^{130}$, computed by $\FF(103)$ below.
\begin{figure}[ht]
\begin{displaymath}
 \begin{array}{c@{\hskip3\cellsize}c@{\hskip3\cellsize}c@{\hskip3\cellsize}c@{\hskip3\cellsize}c@{\hskip3\cellsize}c@{\hskip3\cellsize}c}
 \vline \tableau{  3 & 3 & 3 \\ \\  1 \\ \hline } &   \vline \tableau{  3 & 3 & 2 \\ \\  1 \\ \hline } &   \vline \tableau{  3 & 2 & 2 \\ \\  1 \\ \hline } &   \vline \tableau{  2 & 2 & 2 \\ \\  1 \\ \hline }     \end{array}
\end{displaymath}
\end{figure}

The \emph{monomial slide basis} can also be described using reverse fillings. Given a weak composition $a$, the \emph{monomial fillings} $\MF(a)$ \cite{Sea20} are the subset of $\FF(a)$ for which all entries in the same row are equal. The \emph{monomial slide polynomial} $\msp_a$ \cite{Assaf.Searles} is $$\msp_a = \sum_{T\in \MF(a)} x^{\wt(T)}.$$

For example, $\msp_{103} = x^{103}+x^{130}$.

Various formulas have been given \cite{Assaf.Searles:2}, \cite{Assaf:comb}, \cite{MPS18} for the fundamental slide expansion of a key polynomial. Here we provide another, more in keeping with the theme of the previous section. 

\begin{proposition}\label{prop:keytoslidecompatible}
$$\key_a = \sum_{\rev(b) \sim \col(\keytab(a))} \fs_{\maxcomp(b)}.$$
where $\maxcomp(b)$ is the weak composition associated to the compatible sequence for $b$ whose entries are maximum possible. (If $b$ has no compatible sequences, we declare $\fs_{\maxcomp(b)}=0$.)
\end{proposition}
\begin{proof}
We need to establish $\displaystyle{\fs_{\maxcomp(b)} = \sum_{\textrm{$w$ is $b$-compatible}} x^w}$; the statement then follows from Theorem~\ref{thm:keycompatible}. The compatible sequence for a word $b$ whose entries are maximum possible is found as follows. First, partition $b$ into (weakly) decreasing runs $b=(r_1 | r_2 | \ldots | r_k)$. Let $b^{(i)}$ denote the rightmost (i.e. smallest) entry of $b$ in the $i^{th}$ run $r_i$. We proceed right-to-left, at each step replacing every entry in a run $r_i$ with a certain number $c_i$. To begin, replace every element in $r_k$ with $b^{(k)}$, i.e., we set $c_k=b^{(k)}$. Proceeding leftwards, replace every entry in $r_i$ with $c_i\coloneqq \min\{b^{(i)}, c_{i+1}-1\}$. This process is a variant of the construction of the \emph{weak descent composition} of a word in \cite{Assaf:comb}, \cite{MasSea20}.

Every compatible sequence $w$ for $b$ can be obtained from the maximal one by decrementing parts as long as we still have $w_i<w_{i+1}$ whenever $b_i<b_{i+1}$. In exactly the same way, every fundamental filling for $\maxcomp(b)$ can be obtained from the filling that has every entry equal to its row index by decrementing entries as long as entries in a given row remain strictly larger than entries in any lower row. This gives a weight-preserving bijection between the compatible sequences for $b$ and the fundamental fillings for $\maxcomp(b)$.
\end{proof}

For example, suppose $b=435254$. Then the partition into weakly decreasing runs gives $43|52|54$. We replace each entry in the last run with $4$, obtaining $43|52|{\bf 44}$. Next, we replace each entry in the next run with $\min\{2,4-1\} = 2$, obtaining $43|{\bf 22}|{\bf 44}$. Finally, we replace each entry in the first run with $\min\{3, 2-1\} = 1$, obtaining ${\bf 11}|{\bf 22}|{\bf 44}$. The largest compatible sequence for $b$ is thus $112244$.

\begin{ex}
From the table in Figure~\ref{fig:compseq}, we compute $\key_{032} = \fs_{221} + \fs_{032} + \fs_{131} + \, 0 + \fs_{230}$.
The only compatible sequence for $b=23223$ is $12223$, so $\maxcomp(23223)=131$.
\end{ex}

This yields a formula for the Young fundamental slide expansion of Young key polynomials, proved similarly to Theorem~\ref{thm:ykeycompatible}.

\begin{proposition}\label{prop:ykeytoyslidecompatible}
$$\ykey_a = \sum_{f(b) \sim \col(\keytab(a))} \yfs_{\rev(\maxcomp(b))}.$$
\end{proposition}

\subsection{Quasi-key polynomials and fundamental particles}{\label{sec:SSF}} 
For a weak composition $a$, define the \emph{quasi-key fillings} $\QF(a)$ to be the (reverse) fillings of $D(a)$ satisfying the following conditions.
\begin{enumerate}
\item Entries weakly decrease from left to right in each row.
\item No entry in row $i$ is greater than $i$.
\item Entries strictly increase up the first column, and entries in any column are distinct.
\item All type A and type B triples are inversion triples.
\end{enumerate}
The \emph{quasi-key polynomial} is 
\[\qk_a = \sum_{T\in \QF(a)}x^{\wt(T)}.\]
Quasi-key polynomials were first defined in \cite{Assaf.Searles:2} as a lift of the quasisymmetric Schur functions to a basis of $\Poly_n$. The above formula is due to \cite{MPS18}. 
For example, we have $\qk_{103} = x^{103} +x^{112} +x^{202} +x^{121} +x^{211} +x^{130} +x^{220}$ which is computed by the quasi-key fillings shown in Figure~\ref{fig:QK103} below.
\begin{figure}[ht]
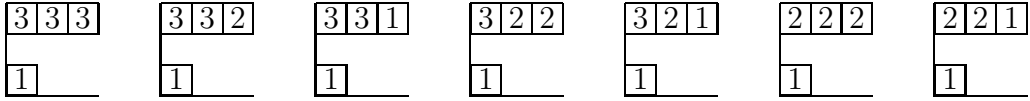

\begin{displaymath}
 \begin{array}{c@{\hskip2\cellsize}c@{\hskip2\cellsize}c@{\hskip2\cellsize}c@{\hskip2\cellsize}c@{\hskip2\cellsize}c@{\hskip2\cellsize}c@{\hskip2\cellsize}c@{\hskip2\cellsize}c@{\hskip2\cellsize}c}
 \vline \tableau{  3 & 3 & 3 \\ \\  1 \\  \hline }  &  \vline \tableau{  3 & 3 & 2 \\ \\  1 \\  \hline }  &   \vline \tableau{  3 & 3 & 1 \\ \\  1 \\  \hline }  &   \vline \tableau{  3 & 2 & 2 \\ \\  1 \\  \hline }  &   \vline \tableau{  3 & 2 & 1 \\ \\ 1 \\  \hline }  & \vline \tableau{  2 & 2 & 2 \\ \\  1 \\  \hline }  &   \vline \tableau{  2 & 2 & 1 \\ \\  1 \\  \hline }       \end{array}
\end{displaymath}
\caption{The 7 quasi-key fillings of shape $103$.}\label{fig:QK103}
\end{figure}

The set of fillings $\ASSF(a)$ generating Demazure atoms is exactly the subset of $\QF(a)$ consisting of those fillings whose entries in the leftmost column are equal to their row index. For example,
$\atom_{103} = x^{103} + x^{112} + x^{202} + x^{121} + x^{211}$, 
which is computed by those fillings in Figure~\ref{fig:QK103} whose leftmost column entries are $1$ and $3$.

Finally, define the \emph{particle fillings} $\LF(a)$ to be the subset of $\ASSF(a)$ consisting of those fillings such that whenever $i<j$, all entries in row $i$ are strictly smaller than all entries in row $j$. Then the \emph{fundamental particle $\fp_a$} \cite{Sea20} is defined to be
\[\fp_a = \sum_{T\in \LF(a)}x^{\wt(T)}.\]

For example, $\fp_{103} = x^{103} + x^{112} + x^{121}$, by the 1st, 2nd, and 4th fillings in Figure~\ref{fig:QK103}.

We give a new formula for $\fp_a$ in terms of compatible sequences. Let $S=\{p_1, p_2, \hdots , p_k \}$ be the set of the partial sums of the entries in $a$ with duplicate entries (obtained when an entry of $a$ is $0$) removed. Then we say that a compatible sequence $w$ for the word formed by writing $a_i$ instances of $i$ consecutively is \emph{$a$-flag compatible} if for all $p_i \in S$, the letter in position $p_i$ of $w$ is equal to the row index of the $i^{th}$ nonzero entry in $a$. 

\begin{theorem}
Let $a$ be a weak composition of length $n$,  Then $$\fp_a = \sum_{w \textrm{ is a-flag compatible}} x^{\comp(w)}.$$
\end{theorem}

\begin{proof}
The statement follows from the fact that the $a$-flag compatible sequences correspond to $\LF(a)$ via the following bijection.  Let $w$ be an $a$-flag compatible sequence and let $\tilde{w}^{(i)}$ be the subword of $w$ corresponding to the subword $a^{(i)}$.  Construct the $i^{th}$ row of a $\LF$ by writing $\tilde{w}^{(i)}$ in weakly decreasing order.  Conditions (1),(2), and (3) in the definition of a $\QF$ are satisfied by construction.  Condition (4) is satisfied since the entries in a given row are all smaller than all of the entries in any higher row.  The flag condition guarantees that these fillings are in $\ASSF(a)$, and further, the fact that the entries in a given row are all smaller than all of the entries in any higher row implies these fillings are in $\LF(a)$.  To obtain an $a$-flag compatible sequence from an element of $\LF(a)$, record the entries in each row from right to left (to force them to be weakly increasing), reading rows from bottom to top.
\end{proof}

Figure~\ref{fig:poset} below shows how the bases discussed here expand into one another. An arrow indicates that the basis at the tail expands positively in the basis at the head. This figure is taken from that in \cite{Sea20}.

\begin{figure}[ht]
\[
\begin{tikzcd}
 \key_a \arrow[rr,"{\rm [AS18]}"] & & \qk_a  \arrow[rr,"{\rm [AS18]}"] \arrow[d, "{\rm [Sea20]}"] & & \fs_a  \arrow[rr,"{\rm [AS17]}"] \arrow[d, "{\rm [Sea20]}"] & & \msp_a \arrow[d] &  \\
 & & \atom_a  \arrow [rr,"{\rm [Sea20]}"] & & \fp_a \arrow[rr] & & x^a  &
\end{tikzcd}
\]
  \caption{Positive expansions between bases defined by reverse fillings.}\label{fig:poset}
\end{figure}
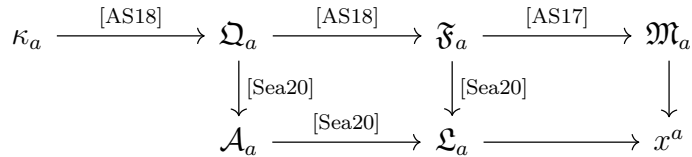

\subsection{Young bases and intersections}

Young analogues may be defined for all the families described above. Indeed, Young analogues of the fundamental slide polynomials and the quasi-key polynomials were introduced and utilized in~\cite{MasSea20}.  In addition to the Young key polynomials and Young Demazure atoms studied in Section 3, Young analogues of the monomial slide polynomials and fundamental particles may be defined similarly, and these families can be shown (by utilizing Lemma~\ref{lem:flipreversediagram} below) to exhibit positive expansions in Figure~\ref{fig:yposet} analogous to those shown in Figure~\ref{fig:poset}.

\begin{figure}[ht]
\[
\begin{tikzcd}
 \ykey_a \arrow[rr] & & \yqk_a  \arrow[rr] \arrow[d] & & \yfs_a  \arrow[rr] \arrow[d] & & \yms_a \arrow[d] &  \\
 & & \ya_a  \arrow [rr] & & \yfp_a \arrow[rr] & & x^a  &
\end{tikzcd}
\]
  \caption{Positive expansions between bases defined by Young fillings.}\label{fig:yposet}
\end{figure}

\begin{remark}\label{rmk:Youngbases}
All of the families of Young polynomials listed in Figure~\ref{fig:yposet} are bases for $\Poly_n$, since their reverse analogues are bases for $\Poly_n$ and the flip-and-reverse process is an involution on $\Poly_n$ that preserves both cardinality and linear independence, cf. Proposition~\ref{prop:youngkeybasis}.
\end{remark}

\begin{lemma}\label{lem:flipreversediagram}
Let $a$ be a weak composition of length $n$, and let $\Fill_a$ denote the set of all possible fillings of $D(a)$ with entries from $1,\ldots , n$, one entry per box. Define $\theta: \Fill_a \rightarrow \Fill_{\rev(a)}$ by letting $\theta(T)$ be the filling obtained by moving all boxes in row $i$ to row $n+1-i$ and replacing every entry $j$ with $n+1-j$, for all $1\le i,j \le n$. Then the following statements are true.
\begin{enumerate}
\item The map $\theta$ is an involution.
\item If $T$ has weight $b$ then $\theta(T)$ has weight $\rev(b)$.
\item The relative order of entries in row $i$ of $T$ is the reverse of the relative order of entries in row $i$ of $\theta(T)$.
\item The relative order of entries in any column of $T$ is the same as the relative order of entries in the same column of $\theta(T)$.
\item A triple of boxes in $T$ is an inversion triple if and only if the image of those boxes is a Young inversion triple in $\theta(T)$.
\end{enumerate} 
\end{lemma}
\begin{proof}
The first four properties are immediate from the definition of $\theta$. Since the relative order of entries in the boxes of a triple in $T$ is the reverse of the relative order of entries in the images of those boxes in $\theta(T)$, it follows from the definition of inversion triples and Young inversion triples that the image of an inversion triple (of type A, respectively B) in $T$ must be a Young inversion triple (of type I, respectively II) in $\theta(T)$. Likewise, the images of non-inversion triples in $T$ are Young non-inversion triples in $\theta(T)$. 
 \end{proof}

Given a weak composition $a$ of length $n$, define the \emph{Young fundamental fillings} $\YFF(a)$ of $a$ to be the fillings of $D(a)$ with entries from $1, \ldots , n$ satisfying the following conditions.
\begin{enumerate}
\item Entries weakly increase from left to right in each row
\item No entry in row $i$ is less than $i$
\item If a box with label $b$ is in a lower row than a box with label $c$, then $b<c$.
\end{enumerate}
In particular, $\YFF(a)$ is the image of $\FF(\rev(a))$ under $\theta$. The \emph{Young fundamental slide polynomial} $\yfs_a$~\cite{MasSea20} is the generating function of $\YFF(a)$:
\[\yfs_a = \sum_{T\in \YFF(a)}x^{\wt(T)}.\]

For example, we have $\yfs_{301}  = x^{301} + x^{211} +x^{121} +x^{031}$, which is computed by the elements of $\YFF(301)$ shown below.
\begin{displaymath}
 \begin{array}{c@{\hskip3\cellsize}c@{\hskip3\cellsize}c@{\hskip3\cellsize}c@{\hskip3\cellsize}c@{\hskip3\cellsize}c@{\hskip3\cellsize}c}
 \vline \tableau{ 3 \\ \\ 1 & 1 & 1 \\ \hline } &   \vline \tableau{ 3 \\ \\ 1 & 1 & 2 \\ \hline } &  \vline \tableau{ 3 \\ \\ 1 & 2 & 2 \\ \hline } &  \vline \tableau{ 3 \\ \\ 2 & 2 & 2 \\ \hline }  \end{array}
\end{displaymath}

For a weak composition $a$ of length $n$, define the \emph{Young monomial fillings} $\YMF(a)$ to be the subset of $\YFF(a)$ for which all entries in any row are equal. Define the \emph{Young monomial slide polynomial} $\yms_a$ to be the generating function of $\YMF(a)$:
\[\yms_a = \sum_{T\in \YMF(a)}x^{\wt(T)}.\]

For example, we have $\yms_{301} = x^{301} +x^{031}$. 

\begin{proposition}
The Young fundamental slide and the Young monomial slide bases of $\mathbb{Z}[x_1,\ldots , x_n]$ contain (respectively) the fundamental quasisymmetric and monomial quasisymmetric bases of quasisymmetric polynomials in $n$ variables. Specifically, if $a$ is a weak composition of length $n$ such that all zero entries are to the right of all nonzero entries, then
\[\yfs_a = F_{\flatten(a)}(x_1,\ldots , x_n) \qquad \mbox{ and } \qquad \yms_a = M_{\flatten(a)}(x_1,\ldots , x_n),\]
where $\flatten(a)$ is the composition obtained by deleting all $0$ parts of $a$. 
\end{proposition}
\begin{proof}
This is shown in \cite{MasSea20} for Young fundamental slides. For monomial slides, since all nonzero entries of $a$ occur before all zero entries, the flag condition on $\YMF$ is always satisfied whenever the other conditions are satisfied. Hence the $\YMF$ are exactly the monomial Young composition tableaux (Proposition~\ref{prop:YoungFisF}).
\end{proof}

\begin{theorem}{\label{thm:intersectionsFM}}
The polynomials in $\mathbb{Z}[x_1,\ldots , x_n]$ that are both a fundamental (respectively, monomial) slide polynomial and a Young fundamental (respectively, monomial) slide polynomial are exactly the fundamental (respectively, monomial) quasisymmetric polynomials in $n$ variables. 

In other words, $\{\fs_a\}\cap \{\yfs_b\} = \{F_\alpha(x_1, \ldots , x_n)\}$ and $\{\msp_a\}\cap \{\yms_b\} = \{M_\alpha(x_1, \ldots , x_n)\}$.
\end{theorem}
\begin{proof}
We prove this in the fundamental case; the monomial case is completely analogous. First, let $\alpha$ be a composition of length $\ell(\alpha) \le n$. Then 
\[F_\alpha(x_1, \ldots x_n) = \fs_{0^{n-\ell(\alpha)}\times \alpha} = \yfs_{\alpha\times 0^{n-\ell(\alpha)}}.\]
For the other direction, let $\fs_a$ be a fundamental slide polynomial that is not equal to $F_\alpha(x_1,\ldots , x_n)$ for any composition $\alpha$. This implies $a$ has a zero entry to the right of a nonzero entry (\cite{Assaf.Searles}). Let $a_j$ be the earliest such zero entry, so $a_{j-1}$ is nonzero. Let $\overline{a}$ denote the weak composition obtained by exchanging the entries $a_{j-1}$ and $a_j$. Then $x^a \in \fs_a$ and $x^{\overline{a}} \notin \fs_a$. However, if a Young fundamental slide polynomial contains $x^a$, it must also contain $x^{\overline{a}}$. Hence $\fs_a$ is not equal to any Young fundamental slide polynomial.
\end{proof}

For a weak composition $a$ of length $n$, define the \emph{Young quasi-key fillings} $\YQF(a)$ to be the (Young) fillings of $D(a)$ obtained by applying $\theta$ to $\QF(\rev(a))$. Specifically, these are the fillings such that entries increase along rows, entries are at least their row index, entries strictly increase up the first column and entries in any column are distinct, and all type I and II Young triples are Young inversion triples.  These generate the \emph{Young quasi-key polynomial} $\yqk_a$ \cite{MasSea20}. Unsurprisingly, the conditions governing the intersections of quasi-key and Young quasi-key polynomials are similar to those governing the intersections of the quasisymmetric bases that they extend (Theorem~\ref{thm:yqsqs}).

\begin{theorem}\label{thm:intersectionsQKYQK}
The polynomials that are both quasi-key and Young quasi-key polynomials are precisely the $\yqk_a$ such that $a$ is a number of equal parts followed by zeros, or a sequence of $1$'s and $2$'s followed by zeros, or $a$ has no zero parts and consecutive parts differ by at most $1$.
\end{theorem} 
\begin{proof}
For any $a$, the polynomial $\yqk_a$ contains the monomial $x^a$, realised by $T\in \YQF(a)$ whose entries in row $j$ are all $j$. Suppose a quasi-key polynomial $\qk_b$ contains $x^a$, realised by some $S\in \QF(b)$. Suppose $a$ has a zero entry preceding a nonzero entry, e.g., $a_i=0$ but $a_{i+1}$ is nonzero. Create $S'$ by changing the rightmost $i+1$ in $S$ to an $i$. Since we change the rightmost $i+1$, entries of $S'$ still decrease along rows, and since no other $i$'s exist in $S$, entries still strictly increase up the first column of $S'$ and do not repeat in any column of $S'$, and the relative order of the entries in any triple in $S$ remains unchanged. Hence $S'\in \QF(b)$, but there is no element of $\YQF(a)$ that has this weight since all entries of $T$ are already minimal possible. Therefore $\yqk_a\neq \qk_b$ for any $b$.

It follows that for $\yqk_a$ to be equal to $\qk_b$, $a$ must consist of an interval of nonzero entries, followed by zero entries. But then $\yqk_a = \yqs_\alpha(x_1, \ldots , x_n)$ by [MS20]. The quasi-key polynomials that are quasisymmetric are exactly the quasisymmetric Schur polynomials: $\qk_b = \qs_\beta(x_1, \ldots , x_n)$ where $b$ is an interval of zero entries followed by an interval $\beta$ of nonzero entries [AS18]. Then, by Theorem~\ref{thm:yqsqs},  $\yqs_\alpha(x_1, \ldots , x_n)$ is equal to $\qs_\beta(x_1, \ldots , x_n)$ exactly when $\alpha$ has all parts the same, or all parts of $\alpha$ are $1$ or $2$, or $\ell(\alpha)=n$ (so $a=\alpha$ has no zero parts) and consecutive parts differ by at most $1$.
\end{proof}

Similarly, define the \emph{Young particle fillings} $\YLF(a)$ to be the image of $\LF(\rev(a))$ under $\theta$. These Young fillings, which are the $\YASSF(a)$ such that any entry in a lower row is strictly smaller than any entry in a higher row, generate the \emph{Young fundamental particle} $\yfp_a$.

\begin{theorem}
The polynomials that are both fundamental particles and Young fundamental particles are precisely the $\yfp_a$ such that $a$ has no zero part adjacent to a part of size at least $2$.
\end{theorem}
\begin{proof}
The $\LF$ (respectively, $\YLF$) obey all the conditions on $\ASSF$ (respectively $\YASSF$), hence the same argument used in the proof of Theorem~\ref{thm:atomyatom} shows that if $\yfp_a = \fp_b$ then $a=b$. 

If $a_{i+1}=0$ and $a_i\ge 2$ for some $i$, then let $T\in \YLF(a)$ such that all entries in each row $j$ are $j$. Let also $T'\in \YLF(a)$ be obtained by changing the rightmost $i$ to $i+1$.  Then there is no $S\in \LF(a)$ with the same weight as $T'$, since the entries in $S$ above row $i$ must agree with those in $T'$ above row $i$, and then there is nowhere the new $i+1$ could be placed in $S$. Hence $\yfp_a\neq \fp_a$. A similar argument shows that if $a_{i+1}\ge 2$ and $a_{i}=0$ then $\fp_a \neq \yfp_a$.

Straightforwardly, $\yfp_a = \fp_a = x^a$ if $a$ has no zero part next to a part of size at least $2$.
\end{proof}

\begin{remark}\label{rmk:noembed}
While the Young and reverse analogues of a given basis have similar definitions, they have important structural differences. Unlike the reverse families, for each family of Young polynomials, the basis of Young polynomials of $\Poly_n$ does not embed into $\Poly_{n+1}$. For example, $\yfs_{0101} = x_2x_4+x_3x_4 \in \Poly_4$ is not a Young fundamental slide polynomial in $\Poly_{5}$. 
Because of this, we cannot use the typical definition of a weak composition as an infinite sequence of nonnegative integers (almost all zero); the number of entries in the sequence matters and the  value of $n$ must be specified.
\end{remark}

\begin{remark}\label{rmk:stablelimit}
Stable limits are obtained for all families in the top row of Figure~\ref{fig:poset} by prepending zeros to the weak composition $a$; they are equal to the appropriate quasisymmetric function for $\flatten(a)$ (\cite[Theorem D]{Sea20}). While stable limits for the Young analogues of these families can be defined (by appending zeros to the weak composition $a$), they are not equal to the quasisymmetric functions for $\flatten(a)$, except in the case that all nonzero entries of $a$ are to the left of all zero entries of $a$. A polynomial satisfying this condition is in fact already quasisymmetric, and indeed limits to the expected quasisymmetric function. 
\end{remark}

For example, the stable limit of the Young fundamental slide polynomial $\yfs_{230}$ (which is equal to $F_{23}(x_1,x_2,x_3)$) is the (Young) fundamental quasisymmetric function $F_{23}$. However, the stable limit of $\yfs_{203}$ is not $F_{23}$.

\section{Young Schubert polynomials}{\label{Sec:Schubert}}

Schubert polynomials were first introduced in \cite{LasSch82} to represent Schubert classes in the cohomology of the flag manifold. Schubert polynomials are typically indexed by permutations.  However, every permutation corresponds to a weak composition called a \emph{Lehmer code}, which may also be used to index the Schubert polynomial.  For each $n$ there is a $\mathbb{Z}$-basis for $\Poly_n$ consisting of Schubert polynomials, however unlike the previously-discussed bases of $\Poly_n$, the indexing compositions of the Schubert basis elements are not compositions of length $n$ but of arbitrary length. It is a long-standing open problem to find a positive combinatorial formula for the structure constants of the Schubert basis.  See~\cite{Mac91,Man98} for more details about the geometry, algebra, and combinatorics of Schubert polynomials.

We will take the combinatorial ``pipe dreams" model introduced in \cite{BerBil93} as our definition of Schubert polynomials.  Consider a permutation $ w \in S_n$.  The \emph{Lehmer code} of $w$ is the weak composition $L(w)$ of length $n$ whose $i^{th}$ term equals the number of pairs $(i,j)$ with $i<j$ such that $w_i > w_j$.  For example, if $w=31254$ then $L(w)=(2,0,0,1,0)$.  A \emph{(reduced) pipe dream} is a tiling of the first quadrant of $\mathbb{Z} \times \mathbb{Z}$ with \emph{elbows} and \emph{crosses} so that any two of the resulting strands (called \emph{pipes}) cross at most once. The associated permutation 
can be read from the diagram by following the pipes from the $y$-axis to the $x$-axis.  Let $\PD(w)$ denote the set of pipe dreams for $w$. The five pipe dreams in $\PD(31524)$ are shown in Figure~\ref{fig:pipedreams}. 

\begin{figure}[ht]
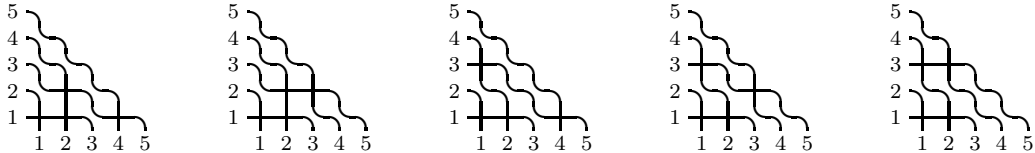


$
\pipes{
               5 & \upelb \\
               4 & \elbow & \upelb \\
               3 & \elbow & \elbow & \upelb \\
               2 & \elbow & \cross & \elbow & \upelb \\
               1 & \cross & \cross & \elbow & \cross & \upelb \\
               &  1 & 2 & 3 & 4 & 5
             } \qquad
 \pipes{
               5 & \upelb \\
               4 & \elbow & \upelb \\
               3 & \elbow & \elbow & \upelb \\
               2 & \elbow & \cross & \cross & \upelb \\
               1 & \cross & \cross & \elbow & \elbow & \upelb \\
               &  1 & 2 & 3 & 4 & 5
             } \qquad
\pipes{
               5 & \upelb \\
               4 & \elbow & \upelb \\
               3 & \cross & \elbow & \upelb \\
               2 & \elbow & \elbow & \elbow & \upelb \\
               1 & \cross & \cross & \elbow & \cross & \upelb \\
               &  1 & 2 & 3 & 4 & 5
             } \qquad
 \pipes{
               5 & \upelb \\
               4 & \elbow & \upelb \\
               3 & \cross & \elbow & \upelb \\
               2 & \elbow & \elbow & \cross & \upelb \\
               1 & \cross & \cross & \elbow & \elbow & \upelb \\
               &  1 & 2 & 3 & 4 & 5
             } \qquad
  \pipes{
               5 & \upelb \\
               4 & \elbow & \upelb \\
               3 & \cross & \cross & \upelb \\
               2 & \elbow & \elbow & \elbow & \upelb \\
               1 & \cross & \cross & \elbow & \elbow & \upelb \\
               &  1 & 2 & 3 & 4 & 5
             }
             $
\caption{The $5$ pipe dreams associated to the permutation $31524$.}{\label{fig:pipedreams}}
\end{figure}

Let $w \in S_n$.  The \emph{Schubert polynomial} $\sch_w = \sch_w(x_1, \hdots  , x_n)$ is given by $$\sch_w = \sum_{P \in \PD(w)} x^{\wt(P)},$$ where $\wt(P)$ is the weak composition whose $i^{th}$ term counts the crosses in the $i^{th}$ row of $P$.

For example, by Figure~\ref{fig:pipedreams} the Schubert polynomial indexed by the permutation $31524$ is $$\sch_{31524} = x_1^3x_2 + x_1^2 x_2^2 + x_1^3 x_3 + x_1^2 x_2 x_3 + x_1^2 x_3^2.$$

Let $\Red(w)$ denote the set of reduced words for a permutation $w$. Every Schubert polynomial can be written as a positive sum of key polynomials according to the following theorem.  

\begin{theorem}[\cite{RS95, LasSch89}]{\label{thm:schubintokeys}}
$$\sch_w = \sum_{\col(T) \in \Red(w^{-1})} \key_{{\rm wt}(K^0_{-}(T))},$$ where the sum is over semistandard Young tableaux $T$, and $K^0_{-}(T)$ is the \emph{left nil key} of $T$, obtained via a modification of Knuth equivalence called \emph{nilplactic} equivalence.

\end{theorem}

\subsection{Young pipe dreams}
Towards giving a combinatorial construction of the Young analogue of Schubert polynomials, we define a Young analogue of pipe dreams. Relabel the row indices (on the $y$-axis) with $n$ as the bottom row, $n-1$ as the second row, and so on. 
Then read the ``reversal" of the permutation by following the pipes from the $y$-axis to the $x$-axis.  This reversal is the permutation $w$ read from right to left (in one-line notation), which we denote $\rev(w)$. This new diagram is called the \emph{Young pipe dream} corresponding to the permutation obtained by reading the pipes in this manner, and the set of all Young pipe dreams for a permutation $w$ is denoted $\YPD(w)$. 
Let the \emph{Young Lehmer code} of a permutation $w\in S_n$, denoted $\YL(w)$, be the weak composition of length $n$ whose $i^{th}$ term is the number of pairs $(i,j)$ with $i>j$ such that $w_i>w_j$.  The \emph{Young weight} $\ywt(P)$ of a Young pipe dream $P$ is the weak composition whose $i^{th}$ part is the number of crosses in the $i^{th}$ row from the top. 

\begin{figure}[ht]
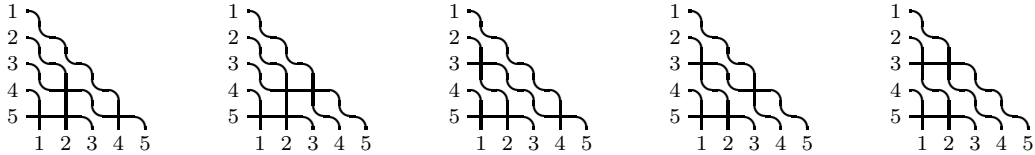


$
\pipes{
               1 & \upelb \\
               2 & \elbow & \upelb \\
               3 & \elbow & \elbow & \upelb \\
               4 & \elbow & \cross & \elbow & \upelb \\
               5 & \cross & \cross & \elbow & \cross & \upelb \\
               &  1 & 2 & 3 & 4 & 5
             } \qquad
 \pipes{
               1 & \upelb \\
               2 & \elbow & \upelb \\
               3 & \elbow & \elbow & \upelb \\
               4 & \elbow & \cross & \cross & \upelb \\
               5 & \cross & \cross & \elbow & \elbow & \upelb \\
               &  1 & 2 & 3 & 4 & 5
             } \qquad
\pipes{
               1 & \upelb \\
               2 & \elbow & \upelb \\
               3 & \cross & \elbow & \upelb \\
               4 & \elbow & \elbow & \elbow & \upelb \\
               5 & \cross & \cross & \elbow & \cross & \upelb \\
               &  1 & 2 & 3 & 4 & 5
             } \qquad
 \pipes{
               1 & \upelb \\
               2 & \elbow & \upelb \\
               3 & \cross & \elbow & \upelb \\
               4 & \elbow & \elbow & \cross & \upelb \\
               5 & \cross & \cross & \elbow & \elbow & \upelb \\
               &  1 & 2 & 3 & 4 & 5
             } \qquad
  \pipes{
               1 & \upelb \\
               2 & \elbow & \upelb \\
               3 & \cross & \cross & \upelb \\
               4 & \elbow & \elbow & \elbow & \upelb \\
               5 & \cross & \cross & \elbow & \elbow & \upelb \\
               &  1 & 2 & 3 & 4 & 5
             }
             $
\caption{The $5$ elements of $\YPD(42513)$.}{\label{fig:youngpipedreams}}
\end{figure}

Let $w\in S_n$. Then the \emph{Young Schubert polynomial} $\ysch_w = \ysch_w(x_1, \hdots  , x_n)$ is given by 
$$\ysch_w = \sum_{P \in \YPD(w)} x^{\ywt(P)}.$$ 

For example, the Young Schubert polynomial associated to the permutation $42513$ can be calculated by reading the Young weights of the Young pipe dreams in Figure~\ref{fig:youngpipedreams} as follows: 
$$\ysch_{42513} = x_4 x_5^3 + x_4^2x_5^2 + x_3 x_5^3 + x_3 x_4 x_5^2 + x_3^2 x_5^2.$$

It is straightforward to check that $\YL(\rev(w)) = \rev(L(w))$.  It follows that
\begin{equation}\label{eqn:SchubertYoungSchubert}
\ysch_w(x_1, \ldots , x_n)=\sch_{\rev(w)}(x_n, \ldots , x_1).
\end{equation}

\begin{remark}
Schubert polynomials have a well-known stability property, namely, for $w\in S_n$, $\sch_w = \sch_{i_n(w)}$, where $i_n:S_n\rightarrow S_{n+1}$ is the embedding in which $S_n$ acts on the first $n$ letters.  The same is not true for Young Schubert polynomials, for example $\ysch_{132} = x_2x_3$ but $\ysch_{1324} = x_2x_3x_4^3$. Analogous to Remark~\ref{rmk:noembed}, a Young Schubert polynomial in $\Poly_n$ is not a Young Schubert polynomial in $\Poly_{n+1}$. Similarly, the stable limit of a Schubert polynomial (on prepending zeros to the Lehmer code) exists, and was shown in \cite{Mac91} to be a \emph{Stanley symmetric function}. Analogous to Remark~\ref{rmk:stablelimit}, there is no corresponding stable limit for Young Schubert polynomials.
\end{remark}

\begin{remark}
The analogue of Remark~\ref{rmk:Youngbases} fails in this case: despite the fact that Schubert polynomials form a basis for $\Poly_n$, no collection of Young Schubert polynomials forms a basis for $\Poly_n$. This is due to the fact that the exponent of $x_i$ in a monomial in a Young Schubert polynomial is bounded by $i-1$. For Schubert polynomials this ``staircase'' condition goes the opposite way: the exponent of $x_i$ is bounded by $m-i$ when the indexing permutation is in $S_m$. Hence, by increasing $m$ as needed, one can find a Schubert polynomial in $\Poly_n$ containing any given monomial in $\Poly_n$.
\end{remark}

\begin{remark}
For completeness, we note that no polynomial is both a Schubert and a Young Schubert polynomial. All Schubert polynomials have at least one monomial divisible by $x_1$, but no Young Schubert polynomials do.
\end{remark}

A permutation $w$ is said to be \emph{vexillary} if for every sequence $a<b<c<d$ of indices, one never has $w_b < w_a<w_d<w_c$.  That is, $w$ is \emph{vexillary} if and only if $w$ avoids the pattern $2143$.  For $w$ vexillary, we have~\cite{LasSch90} $$\sch_w = \key_{L(w)}.$$  Thus the Young Schubert polynomials indexed by permutations whose reversal is vexillary are the Young key polynomials indexed by Young Lehmer codes of $3412$-avoiding permutations.

Theorem~\ref{thm:schubintokeys} and (\ref{eqn:keyyoungkey}) yield the following formula for writing any Young Schubert polynomial as a positive sum of Young key polynomials.
$$\ysch_w = \sum_{\col(T) \in {\rm \Red}((\rev(w))^{-1})} \ykey_{\rev({\rm wt}(K^0_{-}(T)))}.$$

Other combinatorial descriptions of Schubert polynomials can similarly be translated into descriptions of Young Schubert polynomials.

Schubert polynomials were initially defined in terms of divided difference operators so that $$\sch_w(x_1, x_2, \hdots , x_n) = \partial_{w^{-1} w_0} (x_1^{n-1} x_2^{n-2} \cdots x_{n-1}),$$ where $w_0=n \; n-1 \; \cdots 2 \; 1$ is the longest permutation of an $n$-element set and $\partial_i (f) = \frac{f-s_i(f)}{x_i - x_{i+1}}$.  There is a natural way to describe Young Schubert polynomials in terms of divided difference operators, which we establish below. 
For $w\in S_n$, let $\frev(w)$ be the permutation $w_0ww_0$. It is straightforward to see that in one-line notation, $\frev(w)$ is obtained from $w$ by reversing the entries of $w$ and replacing each entry $i$ with $n+1-i$, e.g. $\frev(31542) = 42153$. 


\begin{lemma}\label{lem:reducedwordfrev}
Let $s_{i_1}\cdots s_{i_r}$ be a reduced word for $w\in S_n$. Then $\frev(w) = s_{n-i_1}\cdots s_{n-i_r}$.
\end{lemma}
\begin{proof}
We induct on the length of $w$. If $w$ has length $0$, then $w = \frev(w) = id$ and the statement holds. Now suppose the statement holds for all $w$ of length $r$, for some $r\ge 0$. Suppose $w$ has an ascent in position $j$, i.e. $w(j)<w(j+1)$. Then $w s_j$ has length $r+1$, and is obtained by exchanging the $j$th and $(j+1)$th entries of $w$. We have $w s_j = s_{i_1}\cdots s_{i_r} s_j$; we need to show $\frev(w s_j) = s_{n-i_1}\cdots s_{n-i_r}  s_{n-j}$. But $s_{n-i_1}\cdots s_{n-i_r}$ is equal to $\frev(w)$ by the inductive hypothesis, and therefore $s_{n-i_1}\cdots s_{n-i_r}  s_{n-j}$ is obtained from $\frev(w)$ by exchanging the entries in the $(n-j)$th and $(n-j+1)$th positions. This permutation is exactly $\frev(w s_j)$.
\end{proof}

\begin{lemma}\label{lem:Ipartial}
Let $f$ be a polynomial in $x_1, \ldots , x_n$, and let $I(f)$ be defined as in Lemma~\ref{lem:reversediff}. Then $I ( \partial_{i_1} \cdots \partial_{i_r} (f)) = (-1)^r \partial_{n-i_1} \cdots \partial_{n-i_r} (I(f))$.
\end{lemma}
\begin{proof}
We show that $I(\partial_i (f)) = -\partial_{n-i} (I(f))$; after which repeated iteration establishes the result. To see this, consider the monomial $x_i^a x_{i+1}^b$ where $a>b$.  (The case where $a<b$ is similar and if $a=b$ then $\partial_i (x_i^a x_{i+1}^b) = 0$.) 
\begin{align*}
I(\partial_i (x_i^a x_{i+1}^b)) = I \left(\frac{x_i^{a}  x_{i+1}^{b} - x_i^b x_{i+1}^a}{x_i-x_{i+1}}\right) & =   \frac{x_{n+1-i}^a x_{n-i}^b - x_{n+1-i}^b x_{n-i}^a}{x_{n+1-i}-x_{n-i}} \\
    & = - \partial_{n-i}( x_{n-i}^{b} x_{n+1-i}^a) 
     = - \partial_{n-i} (I(x_i^a x_{i+1}^b)). \qedhere
\end{align*}
\end{proof}
We are now ready to establish a divided difference formula for $\ysch_w$. The power of $-1$ appearing in the formula below is due solely to the fact that since we begin with $x_2 x_3^2 \cdots x_n^{n-1}$, we apply $\partial_i$ to a polynomial whose power of $x_i$ is smaller than its power of $x_{i+1}$ in each monomial. The power of $-1$ could be defined away by replacing the denominator with $x_{i+1} - x_i$ in the definition of $\partial_i$. 
\begin{theorem}
Let $w\in S_n$. Then $\ysch_{w} (x_1, x_2, \hdots , x_n)= (-1)^{\ell(w)}\partial_{w^{-1}} (x_2 x_3^2 \cdots x_n^{n-1} ).$
\end{theorem}

\begin{proof}
Let $s_{i_1}\cdots s_{i_r}$ be a reduced word for $w^{-1}$. Combining Lemmas~\ref{lem:reducedwordfrev} and \ref{lem:Ipartial}, we have
\[I (\partial_{\frev(w^{-1})} (\sch_{w_0})) = I ( \partial_{n-i_1} \cdots \partial_{n-i_r} (\sch_{w_0})) = (-1)^r \partial_{i_1} \cdots \partial_{i_r} (I(\sch_{w_0})) = (-1)^r\partial_{w^{-1}}(I(\sch_{w_0})).\]
Recall that $\ysch_w=I(\sch_{\rev(w)})$, and in particular $\ysch_{id}=I(\sch_{w_0}) = I(x_1^{n-1}x_2^{n-2}\cdots x_{n-1})$. Note also that $w_0^{-1} = w_0$, that $\ell(w) = \ell(w^{-1})$, and that $ww_0 = \rev(w)$. We therefore have

\begin{align*}
\ysch_w = I(\sch_{\rev(w)}) & = I(\partial_{(\rev(w))^{-1}w_0}(\sch_{w_0})) \\
                                           & = I(\partial_{(ww_0)^{-1}w_0}(\sch_{w_0})) \\
                                           & = I(\partial_{w_0w^{-1}w_0}(\sch_{w_0})) \\
                                           & = I(\partial_{\frev(w^{-1})}(\sch_{w_0})) \\
                                           & = (-1)^{\ell(w)}\partial_{w^{-1}}(I(\sch_{w_0})) \\
                                           & = (-1)^{\ell(w)}\partial_{w^{-1}}(\ysch_{id}) 
                                            = (-1)^{\ell(w)}\partial_{w^{-1}}(x_2x_3^2\cdots x_n^{n-1}).  \qedhere
                                          \end{align*} \end{proof}

\begin{ex}
Let $w=2314 = s_1s_2$. Then $w^{-1} = 3124 = s_2s_1$ and we have

\begin{align*}
\ysch_{2314}  = (-1)^{\ell(2314)}\partial_{(2314)^{-1}} (x_2x_3^2x_4^3) 
                      & = (-1)^2\partial_{(3124)} (x_2x_3^2x_4^3) \\
                     & = \partial_2\partial_1 (x_2x_3^2x_4^3) \\
                     & = \partial_2 \left(\frac{x_2x_3^2x_4^3 - x_1x_3^2x_4^3}{x_1-x_2}\right) \\
                     & = \partial_2(-x_3^2x_4^3) \\
                     & = - \frac{x_3^2x_4^3 - x_2^2x_4^3}{x_2-x_3} 
                      = -(-(x_3+x_2)x_4^3) 
                      = x_3x_4^3 + x_2x_4^3 .
\end{align*}
Compare this to $\sch_{\rev(2314)} = \sch_{4132}$, which is equal to $x_1^3x_2 + x_1^3x_3$.
\end{ex}

\subsection{Demazure crystal structure}

We use the recently developed crystal structure for Stanley symmetric functions~\cite{MorSch16} and the Demazure crystal structure for Schubert polynomials~\cite{AssSch18}  to generate the Demazure crystal structure for Young Schubert polynomials. 

Let $w\in S_n$. Following \cite{MorSch16}, a \emph{reduced factorization} for $w$ is a partition of a reduced word for $w$ into blocks (possibly empty) of consecutive entries such that entries decrease from left to right within each block; let $\RF^\ell(w)$ denote the set of all reduced factorisations of $w$ with $\ell$ blocks. In \cite{MorSch16}, a crystal structure is defined on $\RF^\ell(w)$. Precise definitions of the $e_i$ and $f_i$ operators may be found in \cite[Section 3.2]{MorSch16}. See Figure~\ref{fig:YoungDemazureCrystal} for the crystal structure on $\RF^3(21534)$, with arrows $f_i$ labelled. For our purposes, we need to define the weight $\wt(r)$ of $r\in \RF^{\ell}(w)$ to be the weak composition of length $n$ given by $(0, \ldots , 0, |r^\ell|, |r^{\ell-1}| \ldots |r^1|)$ (as opposed to $(|r^\ell|, |r^{\ell-1}| \ldots |r^1|)$ used in \cite{MorSch16}). In particular we define $\wt(r)$ to begin with $n-\ell$ zeros, e.g., for $(41)()(3)\in \RF^3(21534)$, we have $n=5$, $\ell=3$ and $\wt((41)()(3)) = 00102$. 

Let $\ell$ be the position of the rightmost descent in $w$. Define the \emph{reduced factorisations with Young cutoff} for $w$, denoted $\RFYC(w)$, to be those elements of $\RF^\ell(w)$ such that the smallest entry in the $i^{th}$ block from the left is at least $i$.  See Figure~\ref{fig:YoungDemazureCrystal}, in which the elements of $\RFYC(21534)$ are bolded. Compare this to the \emph{reduced factorisations with cutoff} defined in \cite{AssSch18}.

\begin{theorem}
The Young Schubert polynomial $\ysch_w$ is equal to $\sum_{r\in \RFYC(\rev(w))}x^{\wt(r)}$. Moreover, $\RFYC(w)$ is a union of Demazure crystals, under the convention that we begin with the lowest weight rather than the highest and use the $f_i$ operators.
\end{theorem}
\begin{proof}
In \cite{AssSch18}, a crystal structure isomorphic to that of \cite{MorSch16} is obtained by reversing each reduced factorisation for $w$ (thus obtaining reduced factorisations for $w^{-1}$ partitioned into increasing blocks), and exchanging the roles of $f_i$ with $e_{n-i}$ and $e_i$ with $f_{n-i}$. Restricting this isomorphism to $\RFYC(w)$ gives the set of reduced factorisations with cutoff for $w^{-1}$, of which the weight generating function is $\sch_w$ \cite{AssSch18}. Since this isomorphism is weight-reversing, it follows from (\ref{eqn:SchubertYoungSchubert}) that the weight generating function of $\RFYC(w)$ is $\ysch_{\rev(w)}$. By \cite[Theorem 5.11]{AssSch18}, reduced factorisations with cutoff have a Demazure crystal structure, and the isomorphism implies $\RFYC(w)$ is a union of Demazure truncations of the components of $\RF(w)$, starting with the lowest weight.
\end{proof}

\begin{figure}[ht]
\begin{center}
\begin{tikzpicture}[xscale=1.5,yscale=1.35]
 \node at (0,6) (L6) {$()()(431)$};
  \node at (1,5) (L5) {$()(4)(31)$};
  \node at (0,4) (L4a) {$(4)()(31)$};
  \node at (2,4) (L4b) {$()(43)(1)$};
  \node at (1,3) (L3a) {$(4)(3)(1)$};
  \node at (3,3) (L3b) {$()(431)()$};
  \node at (0,2) (L2a) {$(43)()(1)$};
  \node at (2,2) (L2b) {$(4)(31)()$};
  \node at (1,1) (L1) {$(43)(1)()$};
  \node at (0,0) (L0) {${\bf (431)()()}$};
  
  \node at (7,4) (R4) {$()(1)(43)$};
  \node at (5,3) (R3a) {${\bf (1)()(43)}$};
  \node at (9,3) (R3b) {$()()(41)(3)$};
  \node at (6,2) (R2a) {${\bf (1)(4)(3)}$};
  \node at (8,2) (R2b) {$(4)(1)(3)$};
  \node at (5,1) (R1a) {${\bf (41)()(3)}$};
  \node at (9,1) (R1b) {${\bf (1)(43)()}$};
  \node at (7,0) (R0) {${\bf (41)(3)()}$};

  \draw[thick,->,blue  ] (L6) -- (L5) node[midway,above] {$1$};
  \draw[thick,->,blue  ] (L5) -- (L4b) node[midway,above] {$1$};
  \draw[thick,->,blue  ] (L4b) -- (L3b) node[midway,above] {$1$};
  \draw[thick,->,blue  ] (L4a) -- (L3a) node[midway,above] {$1$};
  \draw[thick,->,blue  ] (L3a) -- (L2b) node[midway,above] {$1$};
  \draw[thick,->,blue  ] (L2a) -- (L1) node[midway,above] {$1$};
  
  \draw[thick,->,blue  ] (R4) -- (R3b) node[midway,above] {$1$};
  \draw[thick,->,blue  ] (R3a) -- (R2a) node[midway,above] {$1$};
  \draw[thick,->,blue  ] (R2a) -- (R1b) node[midway,above] {$1$};
  \draw[thick,->,blue  ] (R1a) -- (R0) node[midway,above] {$1$};

  \draw[thick,->,red ] (L5) -- (L4a) node[midway,above] {$2$};
  \draw[thick,->,red ] (L4b) -- (L3a) node[midway,above] {$2$};
  \draw[thick,->,red ] (L3a) -- (L2a) node[midway,above] {$2$};
  \draw[thick,->,red ] (L3b) -- (L2b) node[midway,above] {$2$};
  \draw[thick,->,red ] (L2b) -- (L1) node[midway,above] {$2$};
  \draw[thick,->,red ] (L1) -- (L0) node[midway,above] {$2$};

  \draw[thick,->,red ] (R4) -- (R3a) node[midway,above] {$2$};
  \draw[thick,->,red ] (R3b) -- (R2b) node[midway,above] {$2$};
  \draw[thick,->,red ] (R2b) -- (R1a) node[midway,above] {$2$};
  \draw[thick,->,red ] (R1b) -- (R0) node[midway,above] {$2$};
\end{tikzpicture}
\caption{\label{fig:YoungDemazureCrystal} The crystal on $\RF^3(21534)$ and the subcrystal $\RFYC(21534)$ ({\bf bold}).}
\end{center}
\end{figure}
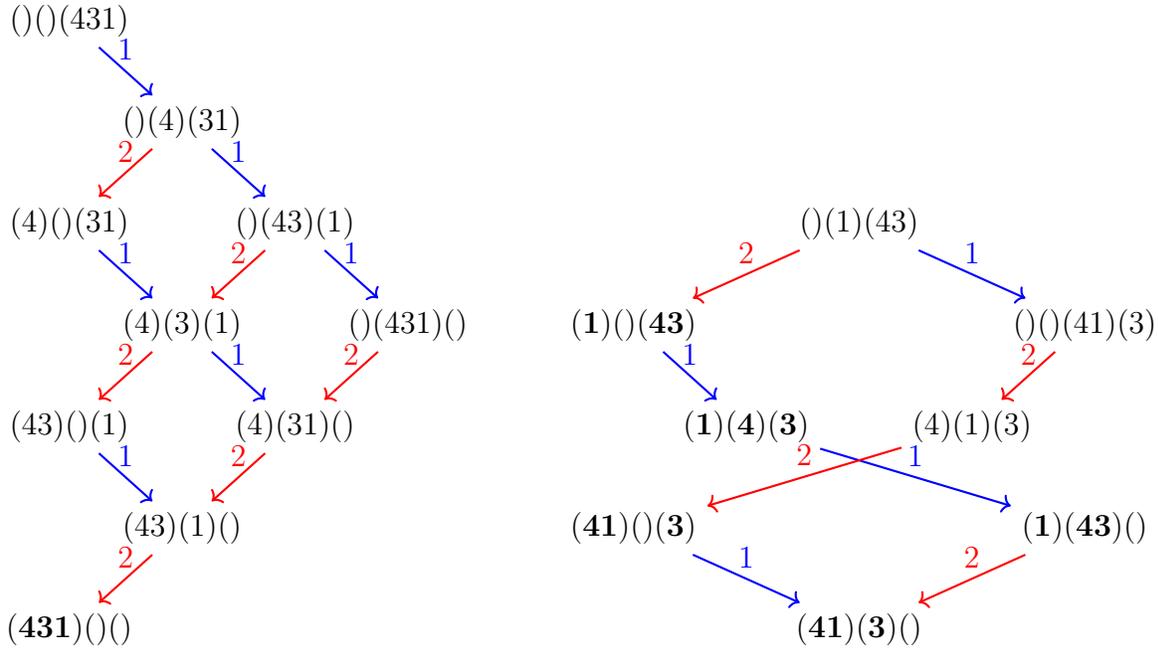

The Demazure crystal structure provides another method for expanding Young Schubert polynomials in Young key polynomials, cf. \cite[Corollary 5.12]{AssSch18}.

\begin{ex}
Figure~\ref{fig:YoungDemazureCrystal} demonstrates that $\ysch_{43512} = \ykey_{00003} + \ykey_{00201}$, where $\ykey_{00003} = x_5^3$ is the bolded Demazure truncation of the left component and $\ykey_{00201} = x_4x_5^2+x_4^2x_5+x_3x_5^2 + x_3x_4x_5 + x_3^2x_5$ is the bolded Demazure truncation of the right component.
\end{ex}

\section*{Acknowledgements}
We thank Sami Assaf and Anne Schilling for suggesting a connection with the Demazure crystal structure for Schubert polynomials, and Martha Precup and Brendon Rhoades for pointing out further recent appearances of the Young/reverse dichotomy for polynomials and tableaux. We also thank Vic Reiner for pointing out a connection to evacuation in Section 3.

\bibliographystyle{alpha}
\bibliography{youngbib}
\label{sec:biblio}

\appendix

\section{Intersections of polynomial families}

In this appendix we determine the polynomials that are both quasisymmetric Schur and Young quasisymmetric Schur polynomials.  Throughout, let $\ell$ be the length of $\alpha$ and $n\ge \ell$ the number of variables.

\begin{lemma}\label{lem:samecomposition}
If $\yqs_\alpha(x_1, \ldots , x_n) = \qs_\beta(x_1, \ldots , x_n)$, then $\alpha = \beta$.
\end{lemma}
\begin{proof}
By the same argument in the proof of Theorem~\ref{thm:atomyatom}, if $\yqs_\alpha = \qs_\beta$, then $\beta$ must be a rearrangement of $\alpha$.  Therefore suppose $\beta$ rearranges $\alpha$ and the length of $\alpha$ (and thus of $\beta$) is $\ell$. Let $T\in \YCT(\alpha)$ be such that the entries in each row $j$ are all $j$. Suppose $S\in \RCT(\beta)$ has the same weight as $T$. Since the first column of $S$ must increase strictly from top to bottom, and we must use all entries $1$ through $\ell$ in $S$, the first entry in each row $j$ of $S$ is forced to be $j$. By the same argument in the proof of Theorem~\ref{thm:atomyatom}, the set of entries in each column of $S$ must be the same as that in the corresponding column of $T$. 

Suppose $\beta\neq \alpha$, and let $i$ be the largest index such that $\beta_i \neq \alpha_i$. Consider rows $\ell$ down to $i+1$, where the row lengths are identical in $\alpha$ and $\beta$. Since entries of $S$ must decrease along rows, the $\ell$'s can only go in the $\ell$th row of $S$, and thus completely fill the $\ell$th row of $S$. By the same reasoning, all $\ell-1$'s must go in row $\ell-1$ of $S$, and so forth down to (and including) row $i+1$. Now, if $\beta_i<\alpha_i$, it is impossible to place $\alpha_i$ many $i$'s in row $i$ of $S$, but $i$'s cannot go in any lower row of $S$ since entries must decrease along rows, and cannot go in any higher row of $S$ since all boxes above row $i$ are occupied, so we cannot construct $S$ of the same weight as $T$. So assume $\beta_i > \alpha_i$. Then we must place $\alpha_i$ many $i$'s in the first $\alpha_i$ boxes of the $i$th row. The next entry placed in this row (in column $\alpha_i+1$) is some $x<i$. Since the column sets of $T$ and $S$ must agree and each column set of $T$ is a subset of the previous one, there must be an $x$ in column $\alpha_i$ of $S$. Since all boxes weakly above row $i$ in this column are occupied by entries at least $i$, $x$ must be strictly below row $i$ in this column. But then these two copies of $x$ must violate one of the triple conditions in $S$. It follows that if $\alpha\neq \beta$, then there is no $S\in \RCT(\beta)$ with the same weight as $T\in \YCT(\alpha)$, and thus $\yqs_\alpha \neq \qs_\beta$. 
\end{proof}

Therefore, the question reduces to determining when $\yqs_\alpha(x_1, \ldots , x_n)=\qs_\alpha(x_1, \ldots , x_n)$.

\begin{lemma}\label{lem:offbytwo}
Let $\alpha$ be a composition of length $\ell$. If there are $i<k$ such that 
\begin{enumerate}
\item $\alpha_i \le \alpha_k-2$ and there is no $i<j<k$ such that $\alpha_j = \alpha_k-1$, or
\item $\alpha_i \ge \alpha_k+2$ and there is no $i<j<k$ such that $\alpha_j = \alpha_i-1$
\end{enumerate}
then $\yqs_\alpha(x_1, \ldots , x_n) \neq \qs_\alpha(x_1, \ldots , x_n)$.
\end{lemma}
\begin{proof}
For (1), create $S\in \RCT(\alpha)$ by letting all entries be equal to their row index, except the last entry of row $k$ is $i$. The condition that there is no $i<j<k$ such that $\alpha_j = \alpha_k-1$ ensures $S$ does not violate the triple condition (B). Then $\wt(S) = (\alpha_1, \ldots \alpha_i+1, \ldots , \alpha_k-1, \ldots \alpha_\ell, 0, \ldots , 0)$. One cannot create $T\in \YCT(\alpha)$ with weight equal to that of $S$. The first column of $T$ must contain the entries $1$ through $\ell$ from bottom to top,  i.e., the first entry of each row is the row index. Then since entries must increase along rows of $T$, all $\alpha_1$ $1$'s must be in row $1$, $\alpha_2$ $2$'s in row $2$, etc, but then one cannot place $\alpha_i+1$ $i$'s in row $i$, since its length is $\alpha_i$. Hence $\yqs_\alpha \neq \qs_\alpha$.
The proof of (2) is similar, starting by creating $T\in\YCT(\alpha)$ whose entries in each row are equal to their row index, except the last entry of row $i$ is $k$.
\end{proof}

It follows from Lemma~\ref{lem:offbytwo} that the only $\alpha$ where $\yqs_\alpha(x_1, \ldots , x_n)$ could possibly be equal to $\qs_\alpha(x_1, \ldots , x_n)$ are those $\alpha$ such that for each $i$, $|\alpha_i-\alpha_{i+1}|\le 1$. 

\begin{lemma}\label{lem:23}
Let $\alpha$ be a composition of length $\ell$ and $n>\ell$. If $2\le \alpha_i < \alpha_{i+1}$ or $2\le \alpha_{i+1} < \alpha_i$ for some $i$, then $\yqs_\alpha(x_1, \ldots , x_n) \neq \qs_\alpha(x_1, \ldots , x_n)$.
\end{lemma}
\begin{proof}
Suppose $2\le \alpha_i < \alpha_{i+1}$. Construct $T\in \YCT(\alpha)$ by letting all entries be equal to their row index in the first $i+1$ rows, except the last entry of row $i$ is $i+2$, and then all entries of each row $r$ for $r>i+1$ are $r+1$. Since $\alpha_i < \alpha_{i+1}$, the triple condition (II) is not violated. Now attempt to construct $S\in \RCT(\alpha)$ with weight equal to that of $T$. All $\ell+1$'s must go in row $\ell$, then all $\ell$'s in row $\ell-1$, down to and including row $i+2$. The sole $i+2$ must be the first entry in row $i+1$, since all boxes above row $i+1$ are occupied. The $i+1$'s can't all fit in row $i+1$, so necessarily the first entry in row $i$ must be $i+1$ if all $i+1$'s are to be placed. This means all $i+1$'s must be placed in row $i+1$ or row $i$. Since $\alpha_i < \alpha_{i+1}$, they cannot all be placed in row $i$; at least one must be in row $i+1$, immediately following the entry $i+2$. But then the $i+1$ in row $i$, column $1$, the $i+2$ in row $i+1$, column $1$, and the $i+1$ in row $i+1$, column $2$ violate the triple condition (B).

For $\alpha$ satisfying $2\le \alpha_{i+1} < \alpha_i$, a similar argument works by letting $S\in \RCT(\alpha)$ be such that entries are equal to to their row index in the first $i-1$ rows, then all entries of each row $r$ for $r \ge i$ are $r+1$, except the last entry of row $i+1$ is $i$.
\end{proof}

\noindent\emph{Proof of Theorem~\ref{thm:yqsqs}:}
It follows from Lemmas~\ref{lem:offbytwo} and \ref{lem:23} that the only $\alpha$ where $\yqs_\alpha$ could possibly be equal to $\qs_\alpha$ are those $\alpha$ whose parts are all the same, those $\alpha$ whose parts are all $1$ or $2$, or (only when $n=\ell(\alpha)$) those $\alpha$ whose consecutive parts differ by at most one. 

If all parts of $\alpha$ are the same, then $\qs_\alpha$ and $\yqs_\alpha$ are both equal to the Schur function $s_\alpha$ by Proposition~\ref{prop:schurexpansion}).

If all parts of $\alpha$ are $1$ or $2$, define a map $\psi$ on tableaux of shape $\alpha$ by swapping the entries in each row of length $2$, and then reordering the rows so the first column is increasing from top to bottom. We will show that $\psi$ restricts to a bijection between $\YCT(\alpha)$ and $\RCT(\alpha)$.  First we observe $\psi$ maps each $T\in \YCT(\alpha)$ to a tableau of shape $\alpha$: if a row of length $1$ is above a row of length $2$ in $T$, then the entry in the row of length $1$ must be larger than both entries of the row of length $2$, the first due to the increasing first column, and the second due to the triple condition (II). If a row of length $1$ is below a row of length $2$, then the entry in the row of length $1$ must be smaller that both entries of the row of length $2$, due to the increasing first column and the fact that entries increase along rows. Hence re-ordering occurs only amongst rows of length $2$ that do not have a row of length $1$ between them. In particular, re-ordering never exchanges a row of length $1$ and a row of length $2$.

Next we show that if $T\in \YCT(\alpha)$, then $\psi(T)$ has no repeated entries in any column. Suppose there are two instances of the same entry $i$ in $T$. The $i$ in column $2$ cannot be strictly above the $i$ in column $1$ because entries increase along rows and strictly increase up the first column. Also, the $i$ in column 2 cannot be strictly below the $i$ in column 1, or these two instances of $i$ would violate one of the triple conditions. Therefore, the $i$'s must be in the same row of $T$, and so cannot be in the same column of $\psi(T)$.

Now we show $\psi(T)\in \RCT(\alpha)$. By definition, entries decrease along rows of $\psi(T)$ and increase up the first column. First consider type B triples in $\psi(T)$, in which case the lower row in the triple has length $1$.  All entries above a given row of length $1$ in $T$ are strictly larger than that entry, since entries increase along rows and up the first column. So the same is true in $\psi(T)$, and the type B triple rule is satisfied. Now consider type A triples in $\psi(T)$. Then both rows in the triple have length $2$.  If these rows are not swapped under $\psi$, then in $T$ the second entry in the higher row is larger than the second entry in the lower row. Combining this with the triple condition in $T$ and the increasing first column, both entries of the higher row must be strictly larger than both entries of the lower row in $T$. 
This implies the same is true in $\psi(T)$, hence the type A triple rule is satisfied.  If they are swapped, we have
\[ T\ni\tableau{ x & y \\ & & \\ z & w} \mapsto \tableau{ w & z \\ & & \\ y & x} \in \psi(T)\]
where $z<x$ and $y<w$. Note that $z<y$, since $x<y$, but also $z<x$, so the triple involving $x,y,z$ in $\psi(T)$ satisfies the type A triple rule. 
Hence the map $\psi$  sends $\YCT(\alpha)$ to $\RCT(\alpha)$. A similar argument shows $\psi$ sends $\RCT(\alpha)$ to $\YCT(\alpha)$ and that $\psi \circ \psi$ is the identity when restricted to either $\RCT(\alpha)$ or $\YCT(\alpha)$, so $\psi:\YCT(\alpha) \rightarrow \RCT(\alpha)$ is a bijection. Since $\psi$ is also weight-preserving, this implies $\yqs_\alpha=\qs_\alpha$.

Finally if consecutive parts of $\alpha$ differ by at most one and $n=\ell(\alpha)$, the only element of $\YCT(\alpha)$ and $\RCT(\alpha)$ is the tableau whose entries in each row $i$ are all $i$, and thus $\yqs_\alpha = x^\alpha = \qs_\alpha$. We proceed by induction on the number of columns of $D(\alpha)$. Suppose $T\in \YCT(\alpha)$; certainly in the first column the entry in each row $i$ must be $i$. Suppose this is true for the first $c$ columns. Consider the boxes in column $c+1$ from highest to lowest. If the highest box $\mathfrak{b}$ is in the top row (i.e., row $n$) then it must have entry $n$ by the increasing row condition. If it is in row $i<n$, then row $i+1$ must be one box shorter than row $i$ (since $\mathfrak{b}$ is highest in its column and consecutive parts of $\alpha$ differ by at most one), and the box in row $i+1$, column $c$ must have entry $i+1$ (by assumption). Then $\mathfrak{b}$ cannot have entry greater than $i$ or the triple condition (II) is violated, so $\mathfrak{b}$ must have entry $i$ by the increasing row condition and the fact that (by assumption) the box immediately left of $\mathfrak{b}$ has entry $i$.

Now suppose the highest $k$ boxes in column $c$ have entry equal to their row index, and suppose the $(k+1)$th highest box $\mathfrak{b}$ is in row $i$. If every row above row $i$ has a box in column $c+1$, then by assumption these boxes all have entry equal to their row index, and then $\mathfrak{b}$ must have entry $i$ since its entry is at least $i$, and entries cannot repeat in a column. Otherwise, consider the lowest row $i'$ above row $i$ that does not have a box in column $c+1$. Since consecutive parts of $\alpha$ differ by at most one, the rightmost box in row $i'$ must be in column $c$, and thus by assumption it has entry $i'$. Then $\mathfrak{b}$ must have entry strictly smaller than $i'$, otherwise triple condition (II) is violated by $\mathfrak{b}$, the box immediately left of $\mathfrak{b}$ (which has entry $i$), and the rightmost box in row $i'$. But since $i'$ was the lowest row above row $i$ without a box in column $c+1$, there are boxes in rows $i+1, \ldots , i'-1$ and column $c+1$ with entry equal to their row index. Therefore, since entries cannot repeat in a column, the entry in $\mathfrak{b}$ must be $i$. 
It follows that all boxes in column $c+1$ have entry equal to their row index, and then that all boxes in $T$ have entry equal to their row index. A similar argument shows that $T$ is also the only element of $\RCT(\alpha)$.
\qed

\end{document}